\documentclass[reqno]{amsart}
\usepackage{amssymb}
\usepackage{array}
\usepackage{graphicx}
\usepackage{caption, booktabs}
\usepackage[all]{xy}
\usepackage{comment}
\usepackage{cite}
\usepackage[left=2.5cm,right=2.5cm,top=3cm,bottom=3cm]{geometry}
\usepackage{color}
\usepackage{comment}

\newcommand{\C}{\mathbb{C}}
\newcommand{\tr}{\mathrm{tr}}
\newcommand{\g}{\mathfrak{g}}
\newcommand{\bb}{\mathfrak{b}}
\newcommand{\ba}{\mathfrak{b}^a}
\newcommand{\ua}{\mathfrak{u}^a}
\newcommand{\lf}{\mathfrak{l}}
\newcommand{\uu}{\mathfrak{u}}
\newcommand{\h}{\mathfrak{h}}

\newcommand{\p}{\mathfrak{p}}

\newcommand{\greg}{\mathfrak{g}_{\mathrm{reg}}}
\newcommand{\gsreg}{\mathfrak{g}_{\mathrm{sreg}}^{a}}

\newcommand{\sln}{\mathfrak{sl}}
\newcommand{\SLn}{\mathrm{SL}}
\newcommand{\gsing}{\mathfrak{g}_{\mathrm{sing}}}

\newcommand{\ad}{\mathrm{ad}}
\newcommand{\Ad}{\mathrm{Ad}}
\newcommand{\Sing}{\mathrm{Sing}}

\renewcommand{\exp}{\mathrm{exp}}

\newcommand\scalemath[2]{\scalebox{#1}{\mbox{\ensuremath{\displaystyle #2}}}}

\numberwithin{equation}{section}

\newtheorem{thm}{Theorem}[section]

\newtheorem{lem}[thm]{Lemma}
\newtheorem{cor}[thm]{Corollary}
\newtheorem{prop}[thm]{Proposition}
\theoremstyle{definition}
\newtheorem{rem}[thm]{Remark}

\newtheorem{defn}[thm]{Definition}

\begin{document}

\title[On the fibres of Mishchenko--Fomenko systems]{On the fibres of Mishchenko--Fomenko systems}

\author[Peter Crooks]{Peter Crooks}
\author[Markus R\"oser]{Markus R\"oser}
\address[Peter Crooks]{Department of Mathematics, Northeastern University, 360 Huntington Avenue, Boston, MA 02115, USA}
\email{p.crooks@northeastern.edu}
\address[Markus R\"oser]{Fachbereich Mathematik, Universit\"at Hamburg, 20146 Hamburg, Germany}
\email{markus.roeser@uni-hamburg.de}


\begin{abstract}
This work is concerned with Mishchenko and Fomenko's celebrated theory of completely integrable systems on a complex semisimple Lie algebra $\g$. Their theory associates a maximal Poisson-commutative subalgebra of $\mathbb{C}[\g]$ to each regular element $a\in\g$, and one can assemble free generators of this subalgebra into a moment map $F_a:\g\rightarrow\mathbb{C}^b$. This leads one to pose basic structural questions about $F_a$ and its fibres, e.g. questions concerning the singular points and irreducible components of such fibres.

We examine the structure of fibres in Mishchenko--Fomenko systems, building on the foundation laid by Bolsinov, Charbonnel--Moreau, Moreau, and others. This includes proving that the critical values of $F_a$ have codimension $1$ or $2$ in $\mathbb{C}^b$, and that each codimension is achievable in examples. Our results on singularities make use of a subalgebra $\ba\subseteq\mathfrak{g}$, defined to be the intersection of all Borel subalgebras of $\g$ containing $a$. In the case of a non-nilpotent $a\in\g_{\text{reg}}$ and an element $x\in\ba$, we prove the following: $x+[\mathfrak{b}^a,\mathfrak{b}^a]$ lies in the singular locus of $F_a^{-1}(F_a(x))$, and the fibres through points in $\ba$ form a $\mathrm{rank}(\g)$-dimensional family of singular fibres. We next consider the irreducible components of our fibres, giving a systematic way to construct many components via Mishchenko--Fomenko systems on Levi subalgebras $\mathfrak{l}\subseteq\g$. In addition, we obtain concrete results on irreducible components that do not arise from the aforementioned construction. Our final main result is a recursive formula for the number of irreducible components in $F_a^{-1}(0)$, and it generalizes a result of Charbonnel--Moreau. Illustrative examples are included at the end of this paper.               
\end{abstract}

\maketitle

\vspace{-10pt}

{\small\tableofcontents} 

\section{Introduction}
\subsection{Context and main results}
There is a fruitful and well-developed synergy between Lie theory and complete integrability, and this is perhaps best witnessed by the Mishchenko--Fomenko systems on a complex semisimple Lie algebra $\g$. Such systems were formally introduced in the late 1970s \cite{Mishchenko}, and they have received considerable attention in the research literature (e.g. \cite{CRR,AbeCrooks,CrooksRayan,Charbonnel,Arakawa,Molev,Bolsinov,Bolsinov91,BolsinovRemarks,Moreau,Shuvalov,Tarasov,Panyushev,Panyushev2,Vinberg,KostantHessenberg}). Particular emphasis has been placed on the fibres of $F_a:\mathfrak{g}\rightarrow\mathbb{C}^b$, the Mishchenko--Fomenko system determined by a regular element $a\in\g$ and a chosen basis of invariant polynomials on $\g$. While these fibres are known to be pure-dimensional \cite{Moreau}, they are sometimes singular and often admit complicated decompositions into irreducible components. The singularities are partly governed by Bolsinov's description of $F_a$ and its critical points \cite{Bolsinov}, while Charbonnel and Moreau \cite{Charbonnel-Moreau} give significant insight into the aforementioned irreducible components. 

We study the singularities and irreducible components of fibres in Mishchenko--Fomenko systems, building on the foundation laid by Bolsinov \cite{Bolsinov}, Charbonnel--Moreau \cite{Charbonnel-Moreau}, Moreau \cite{Moreau}, and others. To describe our results, let $\mathfrak{g}_{\text{reg}}$ denote the set of regular elements in $\mathfrak{g}$ and consider its complement $\mathfrak{g}_{\text{sing}}:=\mathfrak{g}\setminus\mathfrak{g}_{\text{reg}}$. Bolsinov \cite{Bolsinov} shows the critical points of $F_a$ to be given by $\mathrm{Sing}^a:=\mathfrak{g}_{\text{sing}}+\mathbb{C}a\subseteq\mathfrak{g}$, and we use this to investigate the critical values of $F_a$. More precisely, we obtain the following result.

\begin{thm}\label{Theorem: First theorem}
For all $a\in\g_{\emph{reg}}$, the codimension of the closure $\overline{F_a(\mathrm{Sing}^a)}$ in $\mathbb{C}^b$ is $1$ or $2$.
\end{thm}

We then use examples to show that each codimension is achievable.   

While Theorem \ref{Theorem: First theorem} gives information about the critical values of $F_a$, it has no implications for identifying the smooth and singular fibres. We address this by introducing a subalgebra $\mathfrak{b}^a\subseteq\mathfrak{g}$, defined to be the intersection of all Borel subalgebras of $\mathfrak{g}$ containing $a$. This leads us to prove the following.

\begin{thm} \label{Theorem: Second Theorem}
Assume that $a\in\g_{\emph{reg}}$ is not nilpotent. If $x\in\mathfrak{b}^a$, then $x+[\mathfrak{b}^a,\mathfrak{b}^a]$ is contained in the singular locus of $F_a^{-1}(F_a(x))$. In particular, the fibre $F_a^{-1}(F_a(x))$ is singular.  
\end{thm}

This gives context for considering the fibres $F_a^{-1}(F_a(x))$ appearing above, i.e. the fibres lying over points in $F_a(\mathfrak{b}^a)\subseteq\mathbb{C}^b$. One is motivated to gauge the prevalence of these singular fibres amongst all fibres of $F_a$, which amounts to computing $\dim(F_a(\mathfrak{b}^a))$. Our next result gives this dimension in addition to supplementary facts about $F_a(\mathfrak{b}^a)$. 

\begin{thm} \label{Theorem: Third Theorem}
Let $a=s+n$ be the Jordan decomposition of $a\in\g_{\emph{reg}}$ into a semisimple element $s\in\g$ and a nilpotent element $n\in\g$. Fix a Cartan subalgebra $\h\subseteq\g$ containing $s$ and let $W$ be the Weyl group of $(\g,\h)$. Then $F_a(\mathfrak{b}^a)$ is a smooth, $r$-dimensional, closed subvariety of $\mathbb{C}^b$ with coordinate ring canonically isomorphic to $\mathbb{C}[\h]^{W_s}$, where $r$ is the rank of $\g$ and $W_s$ is the $W$-stabilizer of $s$.
\end{thm} 

Our attention subsequently turns to describing the irreducible components of the fibres of $F_a$. To this end, suppose that $\p$ is a parabolic subalgebra of $\g$ containing $a\in\g_{\text{reg}}$. Let $a=s+n$ be the Jordan decomposition, and choose a Cartan subalgebra $\h\subseteq\p$ containing $s$. One then has a unique $\h$-stable Levi factor $\mathfrak{l}\subseteq\mathfrak{p}$. Let $a_{\mathfrak{l}}\in\mathfrak{l}$ denote the projection of $a$ onto $\mathfrak{l}$ with respect to the decomposition $\mathfrak{p}=\mathfrak{l}\oplus\mathfrak{u}$, where $\mathfrak{u}$ is the nilpotent radical of $\mathfrak{p}$. We establish that $a_{\mathfrak{l}}$ is regular in $\mathfrak{l}$, allowing us to form an appropriate Mishchenko--Fomenko system $F_{a_{\mathfrak{l}}}:\mathfrak{l}\rightarrow\mathbb{C}^{b(\mathfrak{l})}$. We then prove the following fact about the irreducible components of $F_a^{-1}(F_a(x))$ for $x\in\p$. 

\begin{thm}\label{Theorem: Correspondence theorem}
Use the objects and notation described in the previous paragraph, and let $x_{\mathfrak{l}}\in\mathfrak{l}$ denote the projection of $x\in\p$ onto $\mathfrak{l}$. If $Y$ is an irreducible component of $F_{a_{\mathfrak{l}}}^{-1}(F_{a_{\mathfrak{l}}}(x_{\mathfrak{l}}))$ containing $x_{\mathfrak{l}}$, then $Y+\mathfrak{u}$ is an irreducible component of $F_a^{-1}(F_a(x))$ containing $x$ and contained in $\mathfrak{p}$. The associated map
\begin{eqnarray*}
\{\text{irred. comp. $Y\subseteq F_{a_{\mathfrak{l}}}^{-1}(F_{a_{\mathfrak{l}}}(x_{\mathfrak{l}}))$ s.t. $x_{\mathfrak{l}}\in Y$}\} &\to& \{\text{irred. comp. $Z\subseteq F_{a}^{-1}(F_{a}(x))$ s.t. $x\in Z\subseteq \mathfrak{p}$}\} \\
Y &\mapsto & Y+\uu
\end{eqnarray*}
is a bijection.
\end{thm}

This result has an interesting inductive quality, as it reduces certain questions about Mishchenko--Fomenko fibres in $\g$ to questions about such fibres in lower-dimensional reductive Lie algebras. On the other hand, it only constructs irreducible components lying in the union of parabolic subalgebras containing $a$. We are thereby motivated to find fibres $F_a^{-1}(z)$ with the following property: there exists at least one irreducible component $Z\subseteq F_a^{-1}(z)$ that is not contained in any proper parabolic subalgebra of $\g$ containing $a$. We call all such irreducible components \textit{exotic}, and we prove the following result.

\begin{prop}\label{Proposition: Early}
Assume that $\g$ is simple and that $a\in\g_{\emph{reg}}$ is semisimple. Let $\g_a$ be the $\g$-centralizer of $a$, and choose a collection of simple positive roots with respect to the Cartan subalgebra $\g_a$. Denote the resulting positive Borel subalgebra by $\mathfrak{b}\subseteq\g$. Let $\xi\in\g$ be a sum of non-zero negative simple root vectors, one for each negative simple root. If $x\in\xi+\mathfrak{b}$ has a non-zero component in the highest root space, then the fibre $F_a^{-1}(F_a(x))$ has an exotic irreducible component. 
\end{prop}            

By virtue of Tarasov's work \cite{Tarasov}, the above-mentioned subset $\xi+\mathfrak{b}$ is a section of $F_a:\g\rightarrow\mathbb{C}^b$. Proposition \ref{Proposition: Early} therefore gives a $b$-dimensional family of fibres admitting exotic irreducible components. One shortcoming is that this family turns out not to include the zero-fibre $F_a^{-1}(0)$, arguably one of the most natural fibres to study. We address this issue in the context of a specific example.

\begin{prop}\label{Proposition: sl3exotic}
If $\g=\mathfrak{sl}_3(\mathbb{C})$, then $F_a^{-1}(0)$ has an exotic irreducible component for all $a\in\g_{\emph{reg}}$.
\end{prop}

The zero-fibre also features prominently in Charbonnel and Moreau's paper \cite{Charbonnel-Moreau}. These authors take $a\in\g_{\text{reg}}$ to be nilpotent and give a recursive formula for the number of irreducible components in $F_a^{-1}(0)$. We generalize this formula to the case of an arbitrary $a\in\mathfrak{g}_{\text{reg}}$, proving the result below.    

\begin{thm}\label{Theorem: RecFormulaIntro}
For all $a\in\g_{\emph{reg}}$, the number of irreducible components in $F_a^{-1}(0)$ is given by \eqref{Equation:RecursiveFormula}.
\end{thm}

To formulate our last main result, we consider the adjoint group $G$ of $\g$ and the adjoint representation $\mathrm{Ad}:G\rightarrow\operatorname{GL}(\g)$. Each $x\in\g$ thereby determines an adjoint orbit $Gx:=\{\mathrm{Ad}_g(x):g\in G\}\subseteq\g$, and these orbits bear the following relation to fibres of $F_a$.

\begin{thm}\label{Theorem: Fibre description}
If $a\in\g_{\emph{reg}}$ and $x\in\g\setminus\mathrm{Sing}^a$, then there exists a finite subset $\Lambda_a\subseteq\mathbb{C}$ for which
$$F_a^{-1}(F_a(x))=\bigcap_{\lambda\in\Lambda_a}\big(\overline{G(x+\lambda a)}-\lambda a\big)$$ and $$F_a^{-1}(F_a(x))\setminus\mathrm{Sing}^a=\bigcap_{\lambda\in\Lambda_a}\big(G(x+\lambda a)-\lambda a\big).$$
\end{thm}  

While this result is incidental to the overall emphasis of our paper, we believe it to be interesting and worth documenting.

\subsection{Organization}
Section \ref{Section: Lie-theoretic foundations} gathers some of the purely Lie-theoretic ideas on which this paper depends. We begin with Section \ref{Subsection: Conventions}, which is largely devoted to notation, conventions, and classical facts. Section \ref{Subsection: Some basic results} then establishes some results about interactions between invariant polynomials, regular elements, and parabolic subalgebras. Section \ref{subsec:ba} subsequently gives a Lie-theoretic introduction to the subalgebra $\ba\subseteq\g$ mentioned above.

Section \ref{Section: Generalities} is concerned with fundamental properties of Mishchenko--Fomenko systems on $\g$. This section begins with \ref{Subsection: The Mishchenko-Fomenko subalgebra}, which recalls the Mishchenko--Fomenko system $F_a:\g\rightarrow\mathbb{C}^b$ and subalgebra $\mathcal{F}_a\subseteq\mathbb{C}[\g]$ associated to each $a\in\greg$. Section \ref{Subsection: Some elementary results} then connects $F_a$ to the Borel subalgebras of $\g$ that contain $a$. This leads to Section \ref{Subsection: Sections of MF}, where we recall Tarasov's sections of $F_a$ and discuss mild extensions thereof. Section \ref{Subsection: Reductive Lie algebras} subsequently harnesses Moreau's work to consider Mishchenko--Fomenko systems on reductive Lie algebras. Such systems are shown to have pure-dimensional fibres, and the fibre dimensions are expressed in Lie-theoretic terms. Section \ref{Subsection: Alternative generators} next gives a non-standard set of free generators for $\mathcal{F}_a$, and Section \ref{Subsection: Some additional} uses this set to prove Theorem \ref{Theorem: Fibre description}.  

Section \ref{Section: Irreducible components} studies the irreducible components of the fibres of $F_a:\g\rightarrow\mathbb{C}^b$. The proofs of Theorem \ref{Theorem: Correspondence theorem} and related facts constitute Section \ref{Subsection: Parabolic components}. These results are then used in Section \ref{Subsection: A recursive formula} to obtain Theorem \ref{Theorem: RecFormulaIntro}. Section \ref{Subsection: Exotic Components} subsequently discusses exotic irreducible components and proves Proposition \ref{Proposition: Early}. 

Section \ref{Section: Singularities} deals with the singularities of Mishchenko--Fomenko systems. Theorem \ref{Theorem: First theorem} is proved in Section \ref{Subsection: Critical values}, while Section \ref{Subsection: A family of singular fibres} is concerned with the proofs of Theorems \ref{Theorem: Second Theorem} and \ref{Theorem: Third Theorem}.

Section \ref{Sec:Examples} illustrates some of our main results in the cases $\g=\mathfrak{sl}_2(\mathbb{C})$ (Section \ref{subsection:sln2}) and $\mathfrak{g}=\mathfrak{sl}_3(\mathbb{C})$ (Section \ref{Subsection: The case g=sln3}). In addition, Section \ref{Subsection: The case g=sln3} contains the proof of Proposition \ref{Proposition: sl3exotic}.

After Section \ref{Sec:Examples}, we list and describe some of the notation appearing throughout our paper. This is merely a quick and convenient reference for the reader, and should not be viewed as a source of definitions.
   
\subsection*{Acknowledgements}
The authors gratefully acknowledge Anne Moreau for fruitful conversations at the very outset of this project. We also wish to thank Stefan Rosemann for several helpful discussions. The first author is supported by an NSERC Postdoctoral Fellowship. 

\section{Lie-theoretic foundations}\label{Section: Lie-theoretic foundations}

In what follows, we establish and discuss the fundamental Lie-theoretic underpinnings of our work. The objects and notation introduced in Section \ref{Subsection: Conventions} shall remain fixed throughout this paper.  

\subsection{Conventions}\label{Subsection: Conventions}
Let $\g$ be a rank-$r$ complex semisimple Lie algebra with Killing form $\langle\cdot,\cdot\rangle:\g\otimes_{\mathbb{C}}\g\rightarrow\mathbb{C}$, adjoint group $G$, and exponential map $\exp:\g\rightarrow G$. One then has the adjoint representation of $G$ on $\g$, to be denoted by
$$\mathrm{Ad}:G\rightarrow\operatorname{GL}(\g),\quad g\mapsto\mathrm{Ad}_g,\quad g\in G.$$
The Killing form is $\mathrm{Ad}$-invariant and therefore induces a $G$-module isomorphism
\begin{equation}\label{Equation: Module isomorphism}\g\xrightarrow{\cong}\g^*,\quad x\mapsto x^{\vee}:=\langle x,\cdot\rangle,\quad x\in\g.\end{equation} Abusing notation slightly, we denote the inverse isomorphism by 
\begin{equation}\label{Equation: Inverse module isomorphism}\g^*\xrightarrow{\cong}\g,\quad \gamma\mapsto \gamma^{\vee},\quad \gamma\in\g^*.\end{equation}
The canonical Lie--Poisson structure on $\g^*$ thereby corresonds to a Poisson bracket $\{\cdot,\cdot\}$ on $\mathbb{C}[\g]:=\mathrm{Sym}(\g^*)$, defined as follows:
\begin{equation}\label{Equation:PoissonBracket}
\{f_1,f_2\}(x):=\langle x, [df_1(x)^{\vee},df_2(x)^{\vee}]\rangle,\quad f_1,f_2\in\mathbb{C}[\g],\text{ }x\in\g,
\end{equation}
where $df_1(x),df_2(x)\in\g^*$ are the differentials at $x$ of $f_1$ and $f_2$, respectively. Note that the symplectic leaves associated with this Poisson bracket are precisely the adjoint orbits of $G$, i.e. the locally closed subvarieties $$Gx:=\{\mathrm{Ad}_g(x):g\in G\}\subseteq\g,\quad x\in\g.$$ 

Now let
$$\mathrm{ad}:\g\rightarrow\mathfrak{gl}(\g),\quad x\mapsto\mathrm{ad}_x=[x,\cdot],\quad x\in\g$$ denote the adjoint representation of $\g$ on itself. Each $x\in\g$ then determines a $G$-stabilizer $$G_x:=\{g\in G:\mathrm{Ad}_g(x)=x\}$$ and a $\g$-centralizer $$\g_x:=\mathrm{ker}(\mathrm{ad}_x)=\{y\in\g: \mathrm{ad}_x(y)=0\}.$$ The former is a closed subgroup of $G$ with Lie algebra equal to the latter, and $x$ is called \textit{regular} if these two objects are $r$-dimensional. Let $\g_{\text{reg}}$ denote the set of all regular elements, which is known to be an open, dense, $G$-invariant subvariety of $\g$. Its complement is the closed subvariety $\g_{\text{sing}}:=\g\setminus\g_{\text{reg}}$ of all \textit{singular} elements in $\g$.   

Recall that $x\in\g$ is called semisimple (resp. nilpotent) if $\mathrm{ad}_x$ is semisimple (resp. nilpotent) as a vector space endomorphism of $\g$. Every $x\in\g$ has a decomposition of the form $x=s+n$, where $s,n\in\g$ are uniquely determined by the following properties: $s$ is semisimple, $n$ is nilpotent, and $[s,n]=0$. One calls $s$ (resp. $n$) the semisimple (resp. nilpotent) part of $x$, and refers to the statement $x=s+n$ as the \textit{Jordan decomposition} of $x$.

We now discuss some conventions regarding the root space decomposition. If $\h\subseteq\g$ is a Cartan subalgebra, we shall let $\Delta\subseteq\h^*$ denote the associated set of roots. Note that each $\alpha\in\Delta$ determines a one-dimensional root space $$\g_{\alpha}:=\{x\in\g:[h,x]=\alpha(h)x\text{ for all }h\in\h\},$$ and that one has
$$\g=\h\oplus\bigoplus_{\alpha\in\Delta}\g_{\alpha}.$$
Now suppose that we have chosen a Borel subalgebra $\mathfrak{b}\subseteq\g$ containing $\h$. This choice induces a partition of $\Delta$ into positive roots $\Delta_{+}$ and negative roots $\Delta_{-}=-\Delta_{+}$, so that
$$\mathfrak{b}=\h\oplus\bigoplus_{\alpha\in\Delta_{+}}\g_{\alpha}.$$ The nilpotent radical of $\mathfrak{b}$ is then given by
$$\mathfrak{u}=[\mathfrak{b},\mathfrak{b}]=\bigoplus_{\alpha\in\Delta_{+}}\g_{\alpha}.$$

Let us use $\Pi\subseteq\Delta_{+}$ to denote the set of simple roots. The subsets of $\Pi$ are in bijective correspondence with the standard parabolic subalgebras of $\g$, i.e. the parabolic subalgebras $\p\subseteq\g$ satisfying $\mathfrak{b}\subseteq\p$. This bijection associates the subset
$$\Pi_{\p}:=\{\alpha\in\Pi:\mathfrak{g}_{-\alpha}\subseteq\p\}$$ to each standard parabolic subalgebra $\p$. The inverse bijection takes a subset $Q\subseteq\Pi$ to a standard parabolic subalgebra $\p_{Q}$, defined via the following procedure. Let $\Delta_Q$ be the set of roots occurring as $\mathbb{Z}$-linear combinations of elements in $Q$, i.e.
$$\Delta_Q:=\Delta\cap\mathrm{span}_{\mathbb{Z}}(Q)\subseteq\h^*.$$ Consider the reductive subalgebra $$\lf_Q:=\h\oplus\bigoplus_{\alpha\in\Delta_Q}\g_{\alpha}$$ and the nilpotent subalgebra
$$\uu_Q:=\bigoplus_{\alpha\in\Delta_{+}\setminus\Delta_Q}\g_{\alpha}.$$ The vector space direct sum
$$\p_Q:=\lf_Q\oplus\uu_Q$$
is then a standard parabolic subalgebra of $\g$ with Levi factor $\lf_Q$ and nilpotent radical $\uu_Q$.

Let us conclude with a few aspects of invariant polynomials on $\g$. To this end, recall that a polynomial $f\in\mathbb{C}[\g]$ is called \textit{invariant} if $f(\mathrm{Ad}_g(x))=f(x)$ for all $g\in G$ and $x\in\g$. One has the subalgebra $\C[\g]^G\subseteq\mathbb{C}[\g]$ of all invariant polynomials on $\g$, and this subalgebra is known to be generated by $r$-many homogeneous, algebraically independent elements. Let $f_1,\ldots,f_r$ be a fixed choice of such generators, and let $d_1,\ldots,d_r\in\mathbb{Z}_{>0}$ denote their respective homogeneous degrees. If $b$ is the dimension of any Borel subalgebra in $\g$, then one has \begin{equation}\label{eq:2b=dimg +r}
b = \tfrac{1}{2}(\dim(\g) + r) = \sum_{i=1}^r d_i.
\end{equation}

It will be advantageous to briefly consider the map 
\begin{equation}\label{eq:AdQuot}
F := (f_1,\dots,f_r):\g\to \C^r
\end{equation}
and its fibres. We note that the closure of an adjoint orbit $\mathcal{O}\subseteq \g_{\text{reg}}$ is a fibre of $F$, provided that this closure is taken in $\g$. The association $\mathcal{O}\mapsto\overline{\mathcal{O}}$ then defines a bijection from the set of adjoint orbits $\mathcal{O}\subseteq\g_{\text{reg}}$ to the set of fibres of $F$. It follows that \begin{equation}\label{eq:AdQuotFibre}
F^{-1}(F(x)) = \overline{Gx}\quad\text{and}\quad F^{-1}(F(x))\cap\g_{\text{reg}}=Gx
\end{equation} for all $x\in\g_{\text{reg}}$. If $x\in\g_{\text{reg}}$ is semisimple, then $Gx$ is closed in $\g$ and $F^{-1}(F(x)) = Gx=F^{-1}(F(x))\cap\g_{\text{reg}}$. We also note that $F^{-1}(0)$ is the cone of all nilpotent elements in $\g$.

\subsection{Some basic results}\label{Subsection: Some basic results}
We now establish a few Lie-theoretic facts needed in subsequent sections, beginning with the following slight generalization of \cite[Corollary 3.1.43]{Chriss}.

\begin{lem}\label{Lemma:3.1.43.parabolic}
Let $\p\subseteq\g$ be a parabolic subalgebra with nilpotent radical $\mathfrak{u}$. If $f\in \C[\mathfrak{g}]^G$ and $x\in\p$, then $f$ is constant on $x+\mathfrak{u}$. 
\end{lem}
\begin{proof}
Let $U$ be the closed, connected subgroup of $G$ having Lie algebra $\uu$, observing that $U=\exp(\uu)$. Given $u\in U$ and $\xi\in\uu$, let us consider the element $\mathrm{Ad}_u(x+\xi)\in\g$. We have $u=\exp(\eta)$ for some $\eta\in\uu$, so that  
\[
\Ad_u(x+\xi) = x + \xi + \sum_{k=1}^\infty\frac{1}{k!}\ad_\eta^k(x +\xi).
\]
Since $\uu$ is an ideal of $\p$, this calculation establishes that $\mathrm{Ad}_u(x+\xi)\in x+\uu$. It follows that $x+\uu$ is invariant under the adjoint action of $U$.

Now assume that our element $x$ is regular and semisimple. The centralizer $\mathfrak{g}_x$ is then a Cartan subalgebra, and as such it must consist of semisimple elements. The subalgebra $\uu$ consists entirely of nilpotent elements, implying that $\mathfrak{g}_x\cap\uu = \{0\}$. Noting again that $\uu$ is an ideal in $\p$, we conclude that $\mathrm{ad}_x$ restricts to a vector space isomorphism
\[
\ad_x: \uu\xrightarrow{\cong}\uu.
\]
This combines with the $U$-invariance of $x+\uu$ and \cite[Lemma 1.4.12]{Chriss} to imply the following: $Ux$ is a Zariski-dense subset of $x+\uu$, where $Ux\subseteq\g$ denotes the $U$-orbit of $x$. Since an orbit of a unipotent group on an affine variety is Zariski-closed (see \cite[Lemma 3.1.1]{Chriss}), we deduce that 
\[
x+\uu = Ux.
\]
Our invariant polynomial $f$ must therefore satisfy 
\[
f(x + \xi) = f(x)
\]
for all $\xi\in\uu$. This proves the lemma in the case that $x$ is regular and semisimple. At the same time, the regular semisimple elements of $\g$ in $\p$ form a dense subset of $\p$. Continuity thereby forces $f$ to be constant on $x+\uu$ for all $x\in \p$.
\end{proof}

\begin{rem}
Note that Lemma \ref{Lemma:3.1.43.parabolic} reduces to \cite[Corollary 3.1.43]{Chriss} if $\mathfrak{p} = \mathfrak{b}$ is a Borel subalgebra. In this particular case, our proof becomes essentially identical to the one given for \cite[Lemma 3.1.44]{Chriss}. 
\end{rem}

The proof of Lemma \ref{Lemma:3.1.43.parabolic} turns out to imply the following extra result.

\begin{cor}
Let $\p\subseteq \g$ be a parabolic subalgebra with nilpotent radical $\uu$ and chosen Levi factor $\lf$, and let $x\in\p$ be such that $\g_x\cap\uu = \{0\}$. Write $x = x_{\lf} + x_{\uu}$ with $x_{\lf}\in\lf$ and $x_{\uu}\in\uu$. Then $$x+\uu = x_{\lf}+\uu = Ux$$ is a single $U$-orbit. If $x$ is also regular and semisimple, then $x_{\mathfrak{l}}$ is regular and semisimple in $\lf$.
\end{cor}

\begin{proof}
The proof of Lemma \ref{Lemma:3.1.43.parabolic} shows that if $x\in\p$ satisfies $\mathfrak{g}_x\cap\uu = \{0\}$, then $x+\uu$ is a single $U$-orbit. Hence $x_{\lf} = x-x_{\uu}\in x+\uu$ must be $U$-conjugate to $x$. If $x$ is regular and semisimple in $\g$, then the previous sentence forces $x_{\lf}$ to be regular and semisimple in $\g$. It is then straightforward to show that $x_{\lf}$ is regular and semisimple in $\lf$. 
\end{proof}

\begin{lem}\label{Lemma:projregnilpotent}
Let $\p\subseteq\g$ be a parabolic subalgebra with nilpotent radical $\uu$ and chosen Levi factor $\lf$, so that $\p=\lf\oplus\uu$. Assume that $a\in\p$ is regular and nilpotent in $\g$, and let $a_{\lf}\in\lf$ denote the projection of $a$ onto $\lf$. Then $a_{\lf}$ is regular in $\lf$.
\end{lem}
\begin{proof}
Let $\bb_\lf\subseteq\lf$ be a Borel subalgebra. Note that $a_{\lf}$ is $L$-conjugate to a point in $\bb_{\lf}$, and that $L$ respects the decomposition $\p=\lf\oplus\uu$. We may therefore assume that $a_{\lf}\in\bb_{\lf}$, in which case $a\in \bb:= \bb_{\lf}\oplus\uu$. Now let $\h$ be a Cartan subalgebra of $\g$ contained in $\bb_{\lf}$. Note that $\h$ and $\bb$ (resp. $\h$ and $\bb_{\lf}$) determine sets of roots $\Delta$ (resp. $\Delta_{\lf}$), positive roots $\Delta_{+}$ (resp. $(\Delta_{+})_{\lf}$), and simple roots $\Pi$ (resp. $\Pi_{\lf}$) for $\g$ (resp. $\lf$), and that we have $$\bb = \h\oplus\bigoplus_{\alpha\in\Delta_+}\g_\alpha.$$ Since $a$ is nilpotent and contained in $\bb$, we conclude that $$a = \sum_{\alpha\in\Delta_+}r_\alpha$$ with $r_{\alpha}\in\g_{\alpha}$ for each $\alpha\in\Delta_{+}$. The regularity of $a$ and \cite[Theorem 5.3]{KostantTDS} then imply that $r_{\alpha}\neq 0$ for all $\alpha\in\Pi$. It follows that
$$a_{\lf}=\sum_{\alpha\in(\Delta_{+})_{\lf}}r_{\alpha}$$
with $r_{\alpha}\neq 0$ for each $\alpha\in\Pi\cap(\Delta_{+})_{\lf}=\Pi_{\lf}$. Another application of \cite[Theorem 5.3]{KostantTDS} then shows $a_{\lf}$ to be regular in $\lf$.
\end{proof}

To prepare for what lies ahead, we record the following standard fact (cf. \cite[Proposition 20.7.6]{Tauvel}).

\begin{lem}\label{Lemma: Jordan decomposition Borel}
Let $\p$ be a parabolic subalgebra of $\mathfrak{g}$. If $x\in\p$, then the semisimple and nilpotent parts of $x$ are contained in $\p$. In particular, this holds if $\p = \bb$ is a Borel subalgebra.
\end{lem}

\begin{lem}\label{Lemma:al_regular_standardparabolic}
Let $a\in\greg$ have a Jordan decomposition of $a=s+n$, where $s\in\g$ is semisimple and $n\in\g$ is nilpotent. Suppose that $\h\subseteq\g$ is a Cartan subalgebra containing $s$, and that $\bb\subseteq\g$ is a Borel subalgebra containing $a$ and $\h$. Fix a subset $Q$ of the simple roots determined by $\h$ and $\bb$, and let $\Pi_Q$, $\Delta_Q$, $\lf_Q$, $\uu_Q$, and $\p_Q:=\lf_Q\oplus\uu_Q$ be as described at the end of Section \ref{Subsection: Conventions}. Let $a_{\lf_Q}\in\lf_Q$ be the projection of $a\in\p_Q$ onto $\lf_Q$. Then $a_{\lf_Q}$ is regular in $\lf_Q$.
\end{lem}
\begin{proof}
Lemma \ref{Lemma: Jordan decomposition Borel} implies that $n\in\bb$, while we know that $n\in\g_s$. These considerations force $n$ to take the form
$$n = \sum_{(\alpha\in\Delta_{+})_s}n_\alpha,$$
where $\Delta_{+}$ is the set of positive roots determined by $\h$ and $\bb$, $(\Delta_{+})_s:=\{\alpha\in\Delta_{+}:\alpha(s)=0\}$, and $n_{\alpha}\in\g_{\alpha}$ for all $\alpha\in(\Delta_{+})_s$.
On the other hand, $\g_a = \g_s\cap\g_n$ is the $\g_s$-centralizer of $n$. The regularity of $a$ in $\g$ thus forces $n$ to be regular in $\g_s$. The previous three sentences and \cite[Theorem 5.3]{KostantTDS} then imply that $n_{\alpha}\neq 0$ for all simple roots $\alpha$ of $\g$ satisfying $\alpha(s)=0$. 

Now let $Q$ be a subset of the simple roots. Noting that $s\in\h\subseteq\lf_Q$, we have $a_{\lf_Q} = s + n_{\lf_Q}$ with 
\begin{equation}\label{Equation: Nilp}n_{\lf_Q} = \sum_{\alpha\in(\Delta_{+})_s\cap\Delta_Q}n_\alpha.\end{equation} The element $s$ commutes with all $n_{\alpha}$ appearing above, so that $a_{\lf_Q} = s + n_{\lf_Q}$ is the Jordan decomposition of $a_{\lf_Q}$. We conclude that the $\lf_Q$-centralizer of $a_{\lf_Q}$ equals the centralizer of $n_{\lf_Q}$ in $(\lf_Q)_s$. It will therefore suffice to prove that $n_{\lf_Q}$ is regular in $(\lf_Q)_s$. We thus observe that the simple roots of $(\lf_Q)_s$ are precisely the simple roots of $\g$ that lie in $Q$ and annihilate $s$. At the same time, the previous paragraph implies that $n_{\alpha}\neq 0$ for all $\alpha\in Q$ satisfying $\alpha(s)=0$. It now follows from \eqref{Equation: Nilp} and \cite[Theorem 5.3]{KostantTDS} that $n_{\lf_Q}$ is regular in $(\lf_Q)_s$, completing the proof.
\end{proof}

\begin{prop}\label{Prop:alregular}
Suppose that $a\in\greg$ has a Jordan decomposition of $a=s+n$. Let $\p\subseteq\g$ be a parabolic subalgebra with $a\in\p$, and choose a Cartan subalgebra $\h\subseteq\g$ with $s\in\h\subseteq\p$. Write $\lf$ for the unique $\h$-stable Levi factor of $\p$, and let $\uu$ denote the nilpotent radical of $\p$. Let $a_{\lf}\in\lf$ be the projection of $a$ onto $\l$ with respect to the decomposition $\p=\lf\oplus\uu$. Then $a_\lf$ is regular in $\lf$. 
\end{prop}
\begin{proof}
Choose a Borel subalgebra $\bb_{\lf}$ of $\lf$ that contains $\h$. Since $\p=\lf\oplus\uu$ is an $L$-module decomposition and $a_{\lf}$ is $L$-conjugate to a point in $\bb_{\lf}$, we may assume that $a_{\lf}\in \bb_{\lf}$. It then follows that $a\in \bb:= \bb_{\lf}\oplus \uu$. We also have $\bb\subseteq\p$, so that $\p=\p_Q$ for some subset $Q$ of the simple roots determined by $\h$ and $\bb$. Using the definition of $\lf_Q$ given at the end of Section \ref{Subsection: Conventions}, we see that $\lf_Q$ is $\h$-invariant. It follows that $\lf_Q=\lf$. 

The previous paragraph shows us to be in the situation of Lemma \ref{Lemma:al_regular_standardparabolic}, implying that $a_{\lf}$ is regular in $\lf$.
\end{proof}

\begin{rem}
Let $\p\subseteq\g$ be a parabolic subalgebra with Levi factor $\lf$ and nilpotent radical $\uu$. Our last few results make extensive use of the decomposition $\p=\lf\oplus\uu$ and the induced projection map $\p\rightarrow\lf$. This projection extends to a projection $\g\rightarrow\lf$, defined as follows. Given $a\in\g$, consider the linear functional $a_{\lf}^*\in\lf^*$ defined by $a_\lf^*(l) = \langle a,l\rangle$, $l\in\lf$. Since the Killing form on $\g$ is non-degenerate when restricted to a bilinear form on $\lf$, $a_\lf^*=\langle a_{\lf},\cdot\rangle$ for some unique element $a_\lf\in\lf$. Our projection $\g\rightarrow\lf$ is then defined by $a\mapsto a_{\lf}$. One then readily verifies that this extends the above-mentioned projection $\p\rightarrow\lf$.

In light of Proposition \ref{Prop:alregular}, it is tempting to imagine that $a\in\g_{\text{reg}}$ implies $a_{\lf}\in\lf_{\text{reg}}$. This implication turns out to be false.   
\end{rem}
\subsection{The subalgebra $\ba$ associated with $a\in\greg$}\label{subsec:ba}
To prepare for Section \ref{Section: Singularities}, we now associate a certain subalgebra $\ba\subseteq\g$ to each $a\in\g_{\text{reg}}$. Explaining this association here allows us to avoid a purely Lie-theoretic digression in Section \ref{Section: Singularities}.
     
Let $a\in \mathfrak{g}_{\text{reg}}$ have a Jordan decomposition of $a = s+n$, so that $s\in\g$ is semisimple, $n\in\g$ is nilpotent, and $[s,n]=0$. The centralizer $\mathfrak{g}_s\subseteq\mathfrak g$ is a reductive subalgebra with $\mathrm{rank}(\mathfrak g_s) = r$. We also have the decomposition
$$\g_s=\mathfrak{z}(\g_s)\oplus[\g_s,\g_s],$$
where $\mathfrak{z}(\g_s)$ denotes the centre of $\g_s$. Let us set $l:=\dim(\mathfrak{z}(\g_s))$, observing that the semisimple subalgebra $[\g_s,\g_s]$ has rank $r-l$. Noting that $a$ is regular and $\g_a=\g_s\cap\g_n$, the previous sentence and a simple dimension count imply that $n$ is regular in $[\mathfrak g_s,\mathfrak g_s]$. It follows that $n$ lies in a unique Borel subalgebra $\tilde{\mathfrak b}^a\subseteq [\g_s,\g_s]$, and that the $[\g_s,\g_s]$-centralizer $[\mathfrak g_s,\mathfrak g_s]_n$ is contained in $\tilde{\mathfrak b}^a$. We thus have 
\begin{equation}\label{eq:defba}
\g_a=\mathfrak g_s \cap \mathfrak g_n = \mathfrak{z}(\mathfrak g_s) \oplus [\mathfrak g_s,\mathfrak g_s]_n \subseteq \mathfrak{z}(\mathfrak g_s)\oplus \tilde{\mathfrak b}^a=:\ba.
\end{equation} 
Observe that $a\in\ba$, and that $\ba$ is a Borel subalgebra of $\mathfrak g_s$. 

\begin{lem}\label{Lemma: Unique Borel}
If $a\in\greg$ has a Jordan decomposition of $a=s+n$, then $\ba$ is the unique Borel subalgebra of $\g_s$ that contains $n$. 
\end{lem}
\begin{proof}
Given a Borel subalgebra $\mathfrak{b}\subseteq\mathfrak{g}_s$ containing $n$, note that $\mathfrak{b}\cap[\mathfrak{g}_s,\mathfrak{g}_s]$ is a solvable subalgebra of $[\mathfrak{g}_s,\mathfrak{g}_s]$ containing $n$. We conclude that $\mathfrak{b}\cap[\mathfrak{g}_s,\mathfrak{g}_s]\subseteq\tilde{\mathfrak{b}}^a$, which implies the inclusion \begin{equation}\label{Equation: Inclusion}\mathfrak{z}(\g_s)\oplus(\mathfrak{b}\cap[\mathfrak{g}_s,\mathfrak{g}_s])\subseteq\ba.
\end{equation} 
Now observe that the codimension of $[\mathfrak{g}_s,\mathfrak{g}_s]$ in $\mathfrak{g}_s$ is precisely $l=\dim(\mathfrak{z}(\mathfrak{g}_s))$, so that $$\dim(\mathfrak{b}\cap[\mathfrak{g}_s,\mathfrak{g}_s])\geq \dim(\mathfrak{b})-l.$$ It follows that the left hand side of \eqref{Equation: Inclusion} has dimension at least $\dim(\mathfrak{b})$. Since this left hand side is also contained in $\mathfrak{b}$, it must equal $\mathfrak{b}$. The inclusion \eqref{Equation: Inclusion} thus becomes $\mathfrak{b}\subseteq\ba$. This forces $\mathfrak{b}=\ba$ to hold, as $\bb$ and $\ba$ are both Borel subalgebras of $\g_s$. 
\end{proof}

\begin{lem}\label{Lemma: Borel intersection}
Suppose that $a\in\g_{\text{reg}}$ has a Jordan decomposition of $a=s+n$. If $a$ is contained in a Borel subalgebra $\mathfrak{b}\subseteq\g$, then $\ba=\mathfrak{b}\cap\mathfrak{g}_s$. In particular, $\mathfrak{b}^{a}\subseteq \mathfrak{b}$. 
\end{lem}
\begin{proof}
Lemma \ref{Lemma: Jordan decomposition Borel} implies that $s,n\in\mathfrak{b}$, so that there exists a Cartan subalgebra $\mathfrak{h}$ of $\mathfrak{g}$ satisfying $s\in\mathfrak{h}\subseteq\mathfrak{b}$. Note also that $\mathfrak{b}$ determines the positive and negative roots of $(\mathfrak{g},\mathfrak{h})$, to be denoted $\Delta_{+}\subseteq\mathfrak{h}^*$ and $\Delta_{-}\subseteq\mathfrak{h}^*$, respectively. Now observe that $\mathfrak{h}$ is a Cartan subalgebra of $\mathfrak{g}_s$, allowing us to consider the roots $\Delta_s\subseteq\mathfrak{h}^*$ of $(\mathfrak{g}_s,\mathfrak{h})$. It is straightforward to check that $\Delta_s\cap\Delta_{+}$ and $\Delta_s\cap\Delta_{-}$ form positive and negative roots in $\Delta_s$, respectively. One then readily verifies that $\mathfrak{b}\cap\mathfrak{g}_s$ is the Borel subalgebra of $\mathfrak{g}_s$ corresponding to the positive roots $\Delta_s\cap\Delta_{+}$.  Together with Lemma \ref{Lemma: Unique Borel} and the fact that $n\in \mathfrak{b}\cap\mathfrak{g}_s$, this implies that $\ba=\mathfrak{b}\cap\mathfrak{g}_s$. 
\end{proof}

\begin{prop}\label{Proposition: Intersection of Borels}
If $a\in\mathfrak{g}_{\emph{reg}}$, then $\ba$ is the intersection of all Borel subalgebras of $\mathfrak{g}$ that contain $a$.
\end{prop}

\begin{proof}
Let $\mathfrak{q}$ denote the intersection of Borel subalgebras that contain $a$. The inclusion $\ba\subseteq\mathfrak{q}$ follows immediately from Lemma \ref{Lemma: Borel intersection}. To show the opposite inclusion, we choose a Borel subalgebra $\mathfrak{b}\subseteq \mathfrak{g}$ containing $a$ and let $\mathfrak{h}$, $\Delta_{+}$, $\Delta_{-}$, and $\Delta_s$ be exactly as introduced in the proof of Lemma \ref{Lemma: Borel intersection}. The aforementioned lemma gives $\ba=\mathfrak{b}\cap\mathfrak{g}_s$, which becomes 
\begin{equation}\label{Equation: Decomposition of b_a}\mathfrak{b}^a=\mathfrak{h}\oplus\bigoplus_{\alpha\in\Delta_+\cap\Delta_s}\mathfrak{g}_{\alpha}
\end{equation} 
once $\mathfrak{b}\cap\mathfrak{g}_s$ is decomposed into $\mathfrak{h}$-weight spaces. 

Let $W:=W(\mathfrak{g},\mathfrak{h})$ and $W_s:=W(\mathfrak{g}_s,\mathfrak{h})$ be the Weyl groups of $(\mathfrak{g},\mathfrak{h})$ and $(\mathfrak{g}_s,\mathfrak{h})$, respectively. Recall that the length $l(t)$ of an element $t\in W$ is the number of negative roots that $t$ sends into $\Delta_{+}$. Denote by $u \in W$ and $v\in W_s$ the longest elements in $W$ and $W_s$, respectively. One knows that $u\Delta_{+}=\Delta_{-}$, $l(v) = |\Delta_-\cap\Delta_s|$ and $v(\Delta_{-}\cap\Delta_s)=\Delta_{+}\cap\Delta_s$. This implies that $v\alpha \in \Delta_-$ for each $\alpha\in\Delta_-\setminus\Delta_s$. With these last two sentences in mind, set $w = vu$ and note that
\begin{eqnarray*}
\Delta_+\cap w\Delta_+ &=& \Delta_+\cap v(u\Delta_+) \\
&=& \Delta_+\cap v\Delta_-\\
&=& \Delta_+\cap v\big((\Delta_-\cap \Delta_s) \cup (\Delta_-\setminus\Delta_s)\big)\\
&=& \Delta_+\cap\big(v(\Delta_-\cap \Delta_s) \cup v(\Delta_-\setminus\Delta_s)\big)\\ 
&=& \Delta_+\cap\Delta_s.
\end{eqnarray*}
This calculation and \eqref{Equation: Decomposition of b_a} imply that $\mathfrak{b}\cap w\mathfrak{b}=\ba$, where $w\mathfrak{b}$ is the Borel subalgebra of $\mathfrak{g}$ defined by
\[
w\mathfrak{b}:=\mathfrak{h}\oplus\bigoplus_{\alpha\in w\Delta_{+}}\mathfrak{g}_{\alpha}.
\]
It follows that $a$ lies in both $\mathfrak{b}$ and $w\mathfrak{b}$, implying the inclusion $\mathfrak{q}\subseteq\mathfrak{b}\cap w\mathfrak{b}=\ba$. This completes the proof.
\end{proof}

\begin{cor}\label{Corollary: gaba}
Suppose that $a\in\greg$. If $\bb\subseteq\g$ is a Borel subalgebra containing $a$, then $\mathfrak g_a\subseteq \bb$.
\end{cor} 

\begin{proof}
The statement \eqref{eq:defba} gives $\g_a\subseteq\ba$, while Proposition \ref{Proposition: Intersection of Borels} implies that $\ba\subseteq\bb$. We conclude that $\g_a\subseteq\bb$. 
\end{proof}

\section{Generalities on Mishchenko--Fomenko systems}\label{Section: Generalities}

We devote the next few sections to a general discussion of Mishchenko--Fomenko systems on $\g$. While Sections \ref{Subsection: The Mishchenko-Fomenko subalgebra}--\ref{Subsection: Sections of MF} largely review relevant parts of the literature, Sections \ref{Subsection: Reductive Lie algebras}--\ref{Subsection: Some additional} contain new results.

\subsection{The Mishchenko--Fomenko subalgebra}\label{Subsection: The Mishchenko-Fomenko subalgebra}
Let us fix $a\in\greg$. Given $f\in \C[\g]$ and $\lambda\in\mathbb{C}$, consider the argument-shifted polynomial $f_{\lambda,a}\in\mathbb{C}[\g]$ given by 
\begin{equation}\label{Equation: Shift}
f_{\lambda,a}(x) := f(x+\lambda a),\quad x\in\g.
\end{equation}
Denote by $\mathcal{F}_a\subseteq\mathbb{C}[\g]$ the subalgebra generated by all $f_{\lambda,a}$ with $f\in\mathbb{C}[\g]^G$ and $\lambda\in\mathbb{C}$, i.e. 
\begin{equation}\label{eq:DefMFsubalgebra}
\mathcal F_a:= \left\langle f_{\lambda,a} |\, f\in\C[\g]^G, \lambda \in\C\right\rangle \subseteq \C[\g].
\end{equation}
We refer to $\mathcal{F}_a$ as the \textit{Mishchenko--Fomenko subalgebra} determined by $a$, largely to recognize its origins in the work \cite{Mishchenko}. 

Now recall the generators $f_1,\dots,f_r\in \C[\g]^G$ fixed in Section \ref{Subsection: Conventions}, and let $d_1,\ldots,d_r\in\mathbb{Z}_{>0}$ be their respective homogeneous degrees. The expansion 
\begin{equation}\label{eq:Def_fij}
(f_i)_{\lambda,a}(x)=f_i(x+\lambda a) = f_i(a)\lambda^{d_i}+\sum_{j=0}^{d_i-1}f_{ij}^a(x)\lambda^j,\quad x\in\g,\text{ }\lambda\in\mathbb{C}
\end{equation}
implicitly defines new polynomials $f_{ij}^a\in\C[\g]$ for $i\in\{1,\ldots,r\}$ and $j\in\{0,\ldots,d_i-1\}$. Note that $f_{i0}^a = f_i$ for all $i\in\{1,\ldots,r\}$, while \eqref{eq:2b=dimg +r} implies that total number of polynomials $f_{ij}$ is $b$. These considerations justify our enumerating the $f_{ij}^a$ as $f_1,\ldots,f_b$, where $f_1,\ldots,f_r$ are exactly as fixed in Section \ref{Subsection: Conventions}. Observe that this notation suppresses the dependence on $a$.

It is straightforward to verify that $f_1,\ldots,f_b$ generate $\mathcal{F}_a$ as an algebra. This is an instance of the following more substantial fact (cf. \cite[Theorem 4.2]{Mishchenko}).

\begin{thm}[Mishchenko--Fomenko]\label{Theorem: Mishchenko-Fomenko}
If $a\in\g_{\emph{reg}}$, then $\mathcal{F}_a$ is a Poisson-commutative subalgebra of $\mathbb{C}[\g]$ freely generated by $f_1,\ldots,f_b$. 
\end{thm}
 
\begin{rem}
The freeness part amounts to $f_1,\ldots,f_b$ being algebraically independent in $\mathbb{C}[\g]$. Mishchenko and Fomenko's arguments in \cite{Mishchenko} only establish this algebraic independence for a semisimple element $a\in\g_{\text{reg}}$. Algebraic independence for all $a\in\g_{\text{reg}}$ can be extracted from \cite[Theorem 1.3]{Bolsinov} or \cite[Section 3]{Panyushev}. 
\end{rem}

Theorem \ref{Theorem: Mishchenko-Fomenko} implies that $f_1,\ldots,f_b$ form a completely integrable system on the Poisson variety $\g$, i.e. $f_{r+1},\ldots,f_b$ restrict to form a completely integrable system on each generic adjoint orbit in $\g$ (cf. \cite{CRR,Panyushev,Mishchenko}). To study this system, one often assembles $f_1,\ldots,f_b$ into a map

\begin{equation}\label{eq:DefF_a}
F_a:=(f_1,\ldots,f_b):\g\rightarrow\mathbb{C}^b.
\end{equation}
 We will sometimes refer to $F_a$ as the \textit{Mishchenko--Fomenko map}, and to the fibres of $F_a$ as \textit{Mishchenko--Fomenko fibres}. 

\begin{rem}\label{Remark: Essential properties}
Observe that the map $F_a$ satisfies $$F_a(\Ad_g(x)) = F_{\Ad_{g^{-1}}(a)}(x)$$ for all $g\in G$ and $x\in \g$. In particular, the essential properties of $F_a$ only depend on the adjoint orbit of $a$. 
\end{rem}

\begin{rem}\label{Remark: Independence of generators}
One can verify that 
\begin{equation}\label{Equation: Fibre via algebra} F_a^{-1}(F_a(x))=\{y\in\g:f(x+\lambda a)=f(y+\lambda a)\text{ }\forall f\in\mathbb{C}[\g]^G,\text{ }\lambda\in\mathbb{C}\}\end{equation} for all $x\in\g$, a fact that we use extensively in this paper.  Now suppose that $h_1,\ldots,h_b\in\mathbb{C}[\g]$ is another collection of algebraically independent generators for $\mathcal{F}_a$, and consider the map $$H_a:=(h_1,\ldots,h_b):\g\rightarrow\mathbb{C}^b.$$ It is then a straightforward consequence of \eqref{Equation: Fibre via algebra} that $F_a^{-1}(F_a(x))=H_a^{-1}(H_a(x))$ for all $x\in\g$.
\end{rem}

It will be advantageous to recall Bolsinov's work on the critical points of $F_a$, i.e. the points $x\in\g$ for which $dF_a(x):T_x\g\rightarrow\mathbb{C}$ has rank less than $b$. His result \cite[Proposition 3.1]{Bolsinov} states that the critical points of $F_a$ constitute the subset 
\begin{equation}\label{eq:Def_gsing+Ca}
\Sing^a := \gsing + \C a\subseteq\g.
\end{equation}
This means that the regular points of $F_a$ are given by 
\[
\gsreg := \g\setminus\Sing^a = \{x\in \g\colon x+\mathbb{C}a\subset\greg\}.
\]
Notice that $\gsreg$ is invariant under the dilation action of $\mathbb{C}^{\times}$ on $\g$. This reflects the fact that each component $f_{ij}^a$ of $F_a$ is homogeneous of degree $d_i-j$.

\subsection{Some elementary results about $F_a$}\label{Subsection: Some elementary results}
In what follows, we establish some straightforward facts about the Mishchenko--Fomenko map $F_a$. These facts are exploited in later parts of the paper.
\begin{prop}\label{Proposition: Binvariance}
Suppose that $a\in\greg$ is contained in a Borel subalgebra $\mathfrak b\subseteq\g$. If $B\subseteq G$ is the corresponding Borel subgroup, then the restriction $F_a\big\vert_{\mathfrak b}:\mathfrak{b}\rightarrow\mathbb{C}^b$ is $B$-invariant.
\end{prop}
\begin{proof}
Choose a Cartan subalgebra $\mathfrak h\subseteq \mathfrak b$ and set $\mathfrak{u}:=[\mathfrak{b},\mathfrak{b}]$. We then have
$\mathfrak{b}=\mathfrak{h}\oplus\mathfrak{u}$ as vector spaces, so that each $x\in\mathfrak{b}$ decomposes as $x=x_{\mathfrak{h}}+x_{\mathfrak{u}}$, $x_{\mathfrak{h}}\in\mathfrak{h}$, $x_{\mathfrak{u}}\in\mathfrak{u}$.
Now let $T\subseteq B$ (resp. $U\subseteq B$) be the maximal torus (resp. maximal unipotent subgroup) corresponding to $\mathfrak{h}$ (resp. $\mathfrak{u}$). It follows that $B = TU$, so that each $b\in B$ takes the form $b=tu$, $t\in T$, $u\in U$. Since the exponential map defines an isomorphism from $\mathfrak{u}$ to $U$, we may write $u=\exp(y)$ for some $y\in\mathfrak{u}$.

Now suppose that $b\in B$ and $x\in\mathfrak{b}$. The previous paragraph justifies our writing $b=t\exp(y)$, $t\in T$, $y\in\mathfrak{u}$ and $x=x_{\mathfrak{h}}+x_{\mathfrak{u}}$, $x_{\mathfrak{h}}\in\mathfrak{h}$, $x_{\mathfrak{u}}\in\mathfrak{u}$. We thus have 
\[
\Ad_{\exp(y)}(x) = x_{\mathfrak{h}}+x_{\mathfrak u} + \sum_{k=1}^\infty \frac{1}{k!}\ad^k_y(x) \in x_{\mathfrak{h}}+\mathfrak{u},
\]
so that 
\[
\Ad_b(x) = \Ad_{t}(\Ad_{\exp(y)}(x))\in\Ad_t(x_{\mathfrak{h}}+\mathfrak{u})=x_{\mathfrak{h}}+\mathfrak{u}.
\]
It follows that any $f\in\C[\mathfrak g]^G$ satisfies
\[
f(\Ad_b(x)+\lambda a) = f(x_{\mathfrak h} + \lambda a) = f(x+\lambda a),
\] 
where we have used \cite[Corollary 3.1.43]{Chriss} twice. We conclude that $F_a(\Ad_b(x))=F_a(x)$ (see \eqref{Equation: Fibre via algebra}), completing the proof.
\end{proof}

\begin{cor}
Let all objects and notation be as set in the statement of Proposition \ref{Proposition: Binvariance}. If $x\in\mathfrak{b}$, then $Bx\subseteq F_a^{-1}(F_a(x))$. In particular, $x+[\bb,\bb]\subseteq F_a^{-1}(F_a(x))$ for all regular semisimple $x\in\bb$. 
\end{cor}

\begin{proof}
Our first assertion is an immediate consequence of Proposition \ref{Proposition: Binvariance}. The second assertion follows from the first, together with the fact that $Bx=x+[\mathfrak{b},\mathfrak{b}]$ for all regular semisimple $x\in\mathfrak{b}$ (see \cite[Lemma 3.1.44]{Chriss}).  
\end{proof}

\subsection{Sections of $F_a$}\label{Subsection: Sections of MF}
Given $a\in\mathfrak{g}_{\text{reg}}$, the following question is natural: does the Mishchenko--Fomenko map $F_a:\g\rightarrow\mathbb{C}^b$ admit a global section, i.e. a closed subvariety $Z\subseteq\g$ for which $F_a\big\vert_Z:Z\rightarrow\mathbb{C}^b$ is a variety isomorphism? Tarasov \cite{Tarasov} provides an affirmative answer for semisimple elements $a\in\mathfrak{g}_{\text{reg}}$ (cf. {\cite[Lemma 4]{Tarasov}}, {\cite[Theorem 3.6]{KostantHessenberg}}).

\begin{thm}\label{Theorem: Section}
Let $a\in\mathfrak{g}_{\emph{reg}}$ be semisimple, and choose a collection $\Pi$ of simple roots for $\mathfrak{g}$ with respect to the Cartan subalgebra $\mathfrak{g}_a$. Let $\mathfrak{b}\subseteq\mathfrak{g}$ denote the positive Borel subalgebra induced by the choice of simple roots, and let $\xi\in\mathfrak{g}$ be of the form
\[
\xi=\sum_{\alpha\in\Pi}e_{-\alpha},\quad e_{-\alpha}\in\mathfrak{g}_{-\alpha}\setminus\{0\},\quad\alpha\in\Pi.
\]
The affine subspace $\xi+\mathfrak{b}$ is then a global section of $F_a$.
\end{thm}

We will sometimes refer to $\xi+\mathfrak{b}$ as a \textit{Tarasov section}. Such sections have interesting implications, some of which we now discuss.

\begin{lem}\label{Lemma: Section}
If $a\in\g_{\emph{reg}}$ and $Z\subseteq\g$ is a global section of $F_a$, then $Z\subseteq\g_{\emph{sreg}}^a$.
\end{lem}  

\begin{proof}
Since $F_a\big\vert_Z$ is an isomorphism, it follows that the differential $d(F_a\big\vert_Z)(z):T_zZ\rightarrow\mathbb{C}^b$ is an isomorphism for all $z\in Z$. We conclude that $dF_a(z):T_z\g\rightarrow\mathbb{C}^b$ has rank $b$ for all $z\in Z$, i.e. every point in $Z$ is a regular point of $F_a$. Since $\g_{\text{sreg}}^a$ is exactly the set of regular points of $F_a$ (see Section \ref{Subsection: The Mishchenko-Fomenko subalgebra}), we must have $Z\subseteq\g_{\text{sreg}}^a$.
\end{proof}  

\begin{cor}\label{Corollary: Surjectivity}
Suppose that $a\in\mathfrak{g}_{\emph{reg}}$ is semisimple, and adopt all notation from Theorem \ref{Theorem: Section}. The following statements then hold. 
\begin{itemize}
\item[(i)] $\xi+\bb\subseteq\g_{\emph{sreg}}^a$;
\item[(ii)] $F_a\big\vert_{\gsreg}:\g_{\emph{sreg}}^a\rightarrow\mathbb{C}^b$ is surjective.
\end{itemize}
\end{cor}

\begin{proof}
Theorem \ref{Theorem: Section} implies that $\xi+\bb$ is a global section of $F_a$. Lemma \ref{Lemma: Section} then yields $\xi+\bb\subseteq\gsreg$, proving (i). Part (ii) follows from (i) and the fact that $\xi+\bb$ is a global section of $F_a$. 
\end{proof}

\begin{rem}
Charbonnel and Moreau \cite[Remark 3]{Charbonnel-Moreau} take $a\in\g_{\text{reg}}$ to be nilpotent and explain that $F_a^{-1}(0)$ need not intersect $\gsreg$ (e.g. $\g = \sln_n(\C)$ for certain $n>3$). By Lemma \ref{Lemma: Section}, this means that $F_a$ need not admit a global section when $a\in\greg$ is nilpotent. One can nevertheless develop the following counterpart of Theorem \ref{Theorem: Section}.

Let $a\in\greg$ be nilpotent, in which case $\ba$ is the unique Borel subalgebra of $\mathfrak{g}$ that contains $a$. The Jacobson--Morozov theorem allows us to include $a$ into an $\sln_2$-triple $(a,h,e)$, meaning that $h,e\in\mathfrak{g}$ satisfy
$$[a,e]=h,\quad [h,a]=2a,\quad\text{and}\quad [h,e]=-2e.$$
Set $\h:=\g_h$, noting that $\h\subseteq\ba$. The Cartan subalgebra $\h\subseteq\ba$ allows us to define an opposite Borel subalgebra $\ba_{-}\subseteq\g$. It is then straightforward to verify that $e\in\ba_{-}$. 

Now suppose that $x\in\h\cap\greg$ and let $B^a\subseteq G$ be the Borel subgroup with Lie algebra $\ba$. Noting that $a\in\ua:=[\ba,\ba]$, \cite[Lemma 3.1.44]{Chriss} implies that
$$x+a+\lambda a\in x+\mathfrak{u}^a=B^ax$$ for all $\lambda\in\mathbb{C}$.
This observation establishes that $a+x\in\gsreg$, so that 
$$\dim(\ker(dF_a(a+x)))=\dim(\mathfrak{g})-b=\dim(\ua).$$
Now note that Lemma \ref{Lemma:3.1.43.parabolic} gives
\[
f(a+x+ \mu v + \lambda a) = f(a+x+\lambda a) 
\]
for any $\mu\in\mathbb{C}$, $v\in\uu^a$, and $f\in\C[\g]^G$, implying that $\uu^a\subseteq \ker(dF_a(a+x))$. These last two sentences imply that 
\[
\uu^a = \ker(dF_a(a+x)).
\]
In particular, the differential of $$F_a\big\vert_{a+\ba_{-}}:a+\ba_{-}\rightarrow\mathbb{C}^b$$ is invertible at $a+x$ for all $x\in\h\cap\greg$.
We may therefore find an open dense neighbourhood $V\subseteq a+\ba_{-}$ of $a + (\h\cap\greg)$ such that $F_a\big\vert_V:V\rightarrow\mathbb{C}^b$ is a local biholomorphism. 

Note that $F_a\big\vert_V$ need not be injective. To see this, suppose that $f\in\C[\g]^G$, $x\in \h\cap\greg$, and $w\in W$, the Weyl group of $(\g,\h)$. Lemma \ref{Lemma:3.1.43.parabolic} then yields
$$f(a+wx + \lambda a) = f(wx) = f(x) = f(a+x+\lambda a),$$
which implies that $F_a(a+wx) = F_a(a+x)$. We thus expect $F_a\big\vert_{a+\bb^a_-}: a+\bb^a_-\to\C^b$ to be generically $\vert W\vert$-to-one as a map to its image. So while there does not seem to be a direct analogue of the Tarasov section in the nilpotent case, we have an affine subspace that intersects generic fibres in finitely many points. 
\end{rem}

\begin{rem}
Our previous remark discusses the non-existence of global sections for certain Mishchenko--Fomenko maps. There turns out to be an interesting counterpart in the context of Gelfand--Zeitlin theory. In more detail, Colarusso and Evens \cite[Corollary 5.19]{Colarusso} prove that the Gelfand--Zeitlin map on $\mathfrak{so}_n(\mathbb{C})$ admits no global sections for $n>3$.
\end{rem}

\subsection{Mishchenko--Fomenko fibres in reductive Lie algebras}\label{Subsection: Reductive Lie algebras}
In the interest of what lies ahead, we now establish some general properties of Mishchenko--Fomenko fibres. We begin by considering the part of Section \ref{Subsection: The Mishchenko-Fomenko subalgebra} that precedes Equation \eqref{eq:DefF_a}, noting that everything makes sense if one replaces $\g$ with an arbitrary reductive Lie algebra $\mathfrak{k}$. We may thereby form a Mishchenko--Fomenko map $F_a:\mathfrak{k}\rightarrow\mathbb{C}^{b(\mathfrak{k})}$ for each $a\in\mathfrak{k}_{\text{reg}}$, where $b(\mathfrak{k})$ is the dimension of a Borel subalgebra in $\mathfrak{k}$. 

\begin{prop}\label{Proposition: Irreducible component dimension}
Let $\mathfrak k$ be a complex reductive Lie algebra of rank $\mathrm{rk}(\mathfrak{k})$. 
If $a\in\mathfrak{k}_{\emph{reg}}$, then all fibres of $F_a$ are pure-dimensional with irreducible components of dimension $b(\mathfrak{k})-\mathrm{rk}(\mathfrak{k})$.
\end{prop}

\begin{proof}
We first assume that $\mathfrak{k}$ is semisimple. 
By \cite[Theorem 1.2]{Moreau}, $F_a$ is flat and surjective. It then follows from \cite[Corollary 9.6]{Hartshorne} that every fibre of $F_a$ is pure-dimensional with irreducible components of dimension $\dim(\mathfrak{k})-b(\mathfrak{k})=b(\mathfrak{k})-\mathrm{rk}(\mathfrak{k})$. 

Now assume that $\mathfrak{k}$ is reductive, and
let $\mathfrak{z}(\mathfrak{k})$ be the centre of $\mathfrak{k}$. One has \begin{equation}\label{Equation: Decomposition}\mathfrak{k} = \mathfrak{z}(\mathfrak{k})\oplus[\mathfrak{k},\mathfrak{k}],\end{equation} where $[\mathfrak{k},\mathfrak{k}] =:\mathfrak{l}$ is semisimple with rank $\mathrm{rk}(\mathfrak{l}) = \mathrm{rk}(\mathfrak k) - \dim(\mathfrak{z}(\mathfrak{k}))$. If $a\in\mathfrak{k}_{\text{reg}}$, then we may write $a = a'+a''$ with $a'\in\mathfrak{z}(\mathfrak{k})$ and $a''\in\mathfrak{l}$. Since $\mathfrak{k}_a=\mathfrak{z}(\mathfrak{k})\oplus\mathfrak{l}_{a''}$, we must have $a''\in\mathfrak{l}_{\text{reg}}$.

Let $d:=\dim(\mathfrak{z}(\mathfrak{k}))$ and choose a basis $f_1,\ldots,f_d$ of $\mathfrak{z}(\mathfrak{k})^*$. Let us also choose homogeneous, algebraically independent generators $f_{d+1},\ldots,f_{\mathrm{rk}(\mathfrak{k})}$ of $\mathbb{C}[\mathfrak{l}]^L$, where $L$ is the adjoint group of $L$. Note that the decomposition \eqref{Equation: Decomposition} allows us to regard $f_1,\ldots,f_{\mathrm{rk}(\mathfrak{k})}$ as polynomials on $\mathfrak{k}$. It is then not difficult to verify that $f_1,\ldots,f_{\mathrm{rk}(\mathfrak{k})}$ are homogeneous, algebraically independent generators of $\mathbb{C}[\mathfrak{k}]^K$, where $K$ is the adjoint group of $\mathfrak{k}$. We then have $F_a:\mathfrak{k}\rightarrow\mathbb{C}^{b(\mathfrak{k})}$, the Mishchenko--Fomenko map obtained from $a\in\mathfrak{k}_{\text{reg}}$ and $f_1,\ldots,f_{\mathrm{rk}(\mathfrak{k})}\in\mathbb{C}[\mathfrak{k}]^K$. Let us also consider the Mishchenko--Fomenko map $F_{a''}:\mathfrak{l}\rightarrow\mathbb{C}^{b(\mathfrak{k})-d}$ obtained from $a''\in\mathfrak{l}_{\text{reg}}$ and $f_{d+1},\ldots,f_{\mathrm{rk}(\mathfrak{k})}\in\mathbb{C}[\mathfrak{l}]^L$. If we use \eqref{Equation: Decomposition} to regard the $f_1,\ldots,f_d$ and the components of $F_{a''}$ as polynomials on $\mathfrak{k}$, then $$F_a=(f_1,\ldots,f_d,F_{a''}).$$ Each fibre of $F_a$ must therefore take the form $\{z\}\times F_{a''}^{-1}(w)$ for $z\in\mathfrak{z}(\mathfrak{k})$ and $w\in\mathbb{C}^{b(\mathfrak{k})-d}$. The first paragraph of this proof then implies that every fibre of $F_a$ is pure-dimensional with irreducible components of dimension $b(\mathfrak{l})-\mathrm{rk}(\mathfrak{l})$, where $b(\mathfrak{l})$ (resp. $\mathrm{rk}(\mathfrak{l})$) denotes the dimension of a Borel subalgebra in $\mathfrak{l}$ (resp. the rank of $\mathfrak{l}$). Note also that $b(\mathfrak{l})=b(\mathfrak{k})-d$ and $\mathrm{rk}(\mathfrak{l})=\mathrm{rk}(\mathfrak{k})-d$, so that
$$b(\mathfrak{l})-\mathrm{rk}(\mathfrak{l})=b(\mathfrak{k})-\mathrm{rk}(\mathfrak{k}).$$ This completes the proof.
\end{proof}

\subsection{Alternative generators of $\mathcal{F}_a$}\label{Subsection: Alternative generators}
Returning to the notation and conventions used prior to Section \ref{Subsection: Reductive Lie algebras}, we now introduce an alternative set of algebraically independent generators for the Mishchenko--Fomenko subalgebra $\mathcal{F}_a$. Fix $a\in\g_{\text{reg}}$ and recall the meaning of $f_{\lambda,a}$ for $f\in\mathbb{C}[\g]$ and $\lambda\in\mathbb{C}$. Recall also that $d_1,\ldots,d_r\in\mathbb{Z}_{>0}$ are the homogeneous degrees of $f_1,\ldots,f_r\in\mathbb{C}[\g]^G$, respectively. We have the following lemma.

\begin{lem}\label{Lemma:AlternativeGenerators}
Let $\Lambda = \{\lambda^{(i)}_j\colon i\in\{1,\dots,r\},\, j\in\{0,\dots,d_i-1\}\}\subseteq\mathbb{C}$ be such that $\lambda^{(i)}_0,\ldots, \lambda^{(i)}_{d_i-1}$ are pairwise distinct for each fixed $i\in\{1,\ldots,r\}$. Fix $a\in\g_{\emph{reg}}$ and define 
\[
g_{ij} := (f_i)_{\lambda^{(i)}_j,a} - (f_i)_{\lambda^{(i)}_j,a}(0)\in\C[\g]
\]
for each $i\in\{1,\ldots,r\}$ and $j\in\{0,\ldots,d_i-1\}$. The polynomials $g_{ij}$ then freely generate $\mathcal{F}_a$ as an algebra.
\end{lem}
\begin{proof}
The homogeneity of $f_i$ implies that $(f_i)_{\lambda,a}(0) = f_i(a)\lambda^{d_i}$ for any $i\in\{1,\ldots,r\}$ and $\lambda\in\C$, so that \eqref{eq:Def_fij} yields
$$(f_i)_{\lambda,a}(x) - (f_i)_{\lambda,a}(0) =\sum_{j=0}^{d_i-1}f_{ij}^a(x)\lambda^j,\quad x\in\g.$$
Setting $\lambda=\lambda_{j}^{(i)}$ for $i\in\{1,\ldots,r\}$ and $j\in\{0,\ldots,d_i-1\}$, we obtain
$$g_{ij}=\sum_{k=0}^{d_i-1}(\lambda_j^{(i)})^kf_{ik}^a.$$
This amounts to the statement
\[
(g_{i0},\dots,g_{i(d_i-1)}) = (f_{i0}^a,\dots,f_{i(d_i-1)}^a)\left(\begin{array}{cccc}
1 & 1& \dots &1 \\ \lambda_0^{(i)}&\lambda_1^{(i)}&\dots&\lambda_{d_i-1}^{(i)}\\
(\lambda_0^{(i)})^2 & (\lambda_1^{(i)})^2 &\dots &(\lambda_{d_i-1}^{(i)})^2\\
\vdots &\vdots & \dots & \vdots\\
(\lambda_0^{(i)})^{d_i-1}&\dots & \dots & (\lambda_{d_i-1}^{(i)})^{d_i-1}
\end{array}\right).
\]
for all $i\in\{1,\ldots,r\}$. Now note that the above-defined $d_i\times d_i$ matrix is invertible, owing to the fact that $\lambda^{(i)}_0,\ldots, \lambda^{(i)}_{d_i-1}$ are pairwise distinct. Since the algebra $\mathcal{F}_a$ is freely generated by the $f_{ij}^a$ for $i\in\{1,\ldots,r\}$ and $j\in\{0,\ldots,d_i-1\}$ (see Theorem \ref{Theorem: Mishchenko-Fomenko}), the previous two sentences imply that $\mathcal{F}_a$ is freely generated by the $g_{ij}$ for $i\in\{1,\ldots,r\}$ and $j\in\{0,\ldots,d_i-1\}$.
\end{proof}

\subsection{Some additional properties of Mishchenko--Fomenko fibres}\label{Subsection: Some additional}
The preceding sections allow us to derive some additional results about the fibres of $F_a$. We begin with the following simple observation. 

\begin{lem}\label{Lemma: Contained in closure}
If $a\in\mathfrak{g}_{\emph{reg}}$, then $F_a^{-1}(F_a(x))\subseteq\overline{Gx}$ for all $x\in\mathfrak{g}_{\emph{reg}}$.
\end{lem}

\begin{proof}
Consider the map $F=(f_1,\ldots,f_r):\mathfrak{g}\rightarrow\mathbb{C}^r$, noting that $F_a^{-1}(F_a(x))\subseteq F^{-1}(F(x))$ for all $x\in\mathfrak{g}$. If $x\in\mathfrak{g}_{\text{reg}}$, then $F^{-1}(F(x))=\overline{Gx}$ by \eqref{eq:AdQuotFibre}. It follows that $F_a^{-1}(F_a(x))\subseteq\overline{Gx}$ for all $x\in\greg$. 
\end{proof}

Lemma \ref{Lemma:AlternativeGenerators} allows us to refine Lemma \ref{Lemma: Contained in closure} as follows.

\begin{thm}\label{Proposition: FibreFaIntersectionOrbitTranslate}
Suppose that $a\in\greg$.
\begin{enumerate}
\item[(i)] If $x\in\gsreg$, then 
\begin{equation}\label{eq:FibreFaIntersectionOrbitTranslate1}
F_a^{-1}(F_a(x)) = \bigcap_{\lambda\in\C}\left(\overline{G(x+\lambda a)} - \lambda a\right)
\end{equation}
and
\begin{equation} \label{eq:FibreFaIntersectionOrbitTranslate2}
F_a^{-1}(F_a(x))\cap\gsreg = \bigcap_{\lambda\in\C}\left(G(x+\lambda a) - \lambda a\right).
\end{equation}
\item[(ii)] If $x\in\greg$, then there exists a finite subset $\Lambda\subseteq\mathbb{C}$ with the following property: \eqref{eq:FibreFaIntersectionOrbitTranslate1} and \eqref{eq:FibreFaIntersectionOrbitTranslate2} are true if one only intersects over all $\lambda\in\Lambda$.
\end{enumerate}
\end{thm}
\begin{proof}
We begin with the proof of (i). Suppose that $x,y\in\g$, and recall the polynomials $f_1,\ldots,f_b\in\mathbb{C}[\g]$ from Section \ref{Subsection: The Mishchenko-Fomenko subalgebra}. We see that 
\[
f_{i}(x) = f_{i}(y)\quad \forall i\in\{1,\ldots,b\} \iff f_i(x+\lambda a) = f_i(y+\lambda a)\quad \forall\ i\in\{1,\ldots,r\},\text{ }\lambda\in\C.
\]
Now suppose that $x\in \gsreg$. This means precisely that $x+\C a\subseteq \greg$, which by \eqref{eq:AdQuotFibre} yields $F^{-1}(F(x+\lambda a)) = \overline{G(x+\lambda a)}$ for all $\lambda\in\mathbb{C}$. 

Assuming that $x\in\gsreg$ and using the previous paragraph where apprpriate, we obtain 
\begin{eqnarray*}
F_a^{-1}(F_a(x)) &=& \{y\in\g\colon f_{i}(y) = f_{i}(x)\text{ }\forall i\in\{1,\ldots,b\} \}\\
&=& \{y\in\g\colon f_i(y+\lambda a) = f_i(x+\lambda a)\text{ }\forall i\in\{1,\ldots,r\},\text{ }\lambda\in \C\}\\
&=& \{y\in\g\colon y+\lambda a\in\overline{G(x+\lambda a)}\text{ }\forall \lambda\in\C\}\\
&=& \bigcap_{\lambda\in\C}\left(\overline{G(x+\lambda a)}-\lambda a\right).
\end{eqnarray*}
This verifies \eqref{eq:FibreFaIntersectionOrbitTranslate1}. Using \eqref{eq:AdQuotFibre} and the fact that $y\in\gsreg$ implies $y+\lambda a\in\greg$ for all $\lambda\in\C$, one can obtain \eqref{eq:FibreFaIntersectionOrbitTranslate2} via a similar argument. 

Now we come to the proof of (ii). Note that if $x\in\greg$, then $x+\lambda a\in\gsing$ for only finitely many $\lambda\in\mathbb{C}$. Keeping this in mind, we may find a finite set $\Lambda = \{\lambda^{(i)}_j\colon i\in\{1,\dots,r\},\, j\in\{0,\dots,d_i-1\}\}\subseteq\mathbb{C}$ with the following properties: 
\begin{itemize}
\item $\lambda_0^{(i)}=0$ for all $i\in\{1,\ldots,r\}$;
\item $\lambda_0^{(i)},\ldots,\lambda_{d_i-1}^{(i)}$ are pairwise distinct for all $i\in\{1,\ldots,r\}$;
\item $x+\lambda_j^{(i)}a\in\mathfrak{g}_{\text{reg}}$ for all $i\in\{1,\ldots,r\}$ and $j\in\{0,\ldots,d_i-1\}$.
\end{itemize}
Lemma \ref{Lemma:AlternativeGenerators} then constructs algebraically independent generators $\{g_{ij}\}$ of $\mathcal{F}_a$, and we may use these generators to compute $F_a^{-1}(F_a(x))$ (see Remark \ref{Remark: Independence of generators}). Note also that the condition $\lambda_0^{(i)}=0$ yields $g_{i0}=f_i$ for all $i\in\{1,\ldots,r\}$. The proof of (ii) then proceeds analogously to that of (i), after we replace $f_1,\ldots,f_b$ with the $g_{ij}$.
\end{proof}

We now provide some incidental descriptions of the tangent spaces to Mishchenko--Fomenko fibres. To this end, recall the Poisson bracket \eqref{Equation:PoissonBracket} and all associated notation.

\begin{prop}\label{Lemma:TangentspaceFibre}
Suppose that $a\in\greg$ and $x\in\gsreg$. The tangent space $T_x(F_a^{-1}(F_a(x)))\subseteq\mathfrak{g}$ can then be described as follows.
\begin{enumerate}
\item[(i)]
If $\{h_i\}_{i\in I}\subseteq\mathbb{C}[\g]$ generates $\mathcal F_a$ as an algebra, then 
\begin{equation}\label{eq:TangentspaceFibre1}
T_x(F_a^{-1}(F_a(x))) = \mathrm{span}\{[x,dh_i(x)^\vee]:i\in I\}.
\end{equation}
Taking the components of $F_a$ as our generators (cf. \eqref{eq:DefF_a}), we get
\begin{equation}\label{eq:TangentspaceFibre2}
T_x(F_a^{-1}(F_a(x))) = \mathrm{span}\{[x, df_{i}(x)^\vee]: i\in\{r+1,\ldots,b\}\}.
\end{equation}
\item[(ii)] We have
\begin{equation}\label{eq:TangentspaceFibre3}
T_x(F_a^{-1}(F_a(x))) = \mathrm{span}\{[a,df(x+\lambda a)^{\vee}]: f\in \C[\g]^G,\lambda\in\C^{\times}\}.
\end{equation}
\end{enumerate}
\end{prop}
\begin{proof}
We begin by proving (i). Note that $T_x(F_a^{-1}(F_a(x)))$ is the span of the Hamiltonian vector fields of $f_1,\ldots,f_b$ at $x$. On the other hand, $f_1,\ldots,f_b$ and $\{h_i\}_{i\in I}$ generate the same subalgebra of $\mathcal{F}_a$ of $\mathbb{C}[\g]$. It is then easy to deduce that $T_x(F_a^{-1}(F_a(x)))$ is the span of the Hamiltonian vector fields of the $h_i$ at $x$. Note also that \eqref{Equation:PoissonBracket} forces the Hamiltonian vector field of $h_i$ at $x$ to be $[x,dh_{i}(x)^\vee]$. This proves \eqref{eq:TangentspaceFibre1}. The statement \eqref{eq:TangentspaceFibre2} follows from the first part of (i) and the fact that $[x, df_{i}(x)^\vee]=0$ for all $i\in\{1,\ldots,r\}$ (see \cite[Proposition 1.3]{KostantSolution}). This proves (i). 

We now verify (ii). By applying (i) to the generating set $\{f_{\lambda,a}\colon f\in \C[\g]^G, \, \lambda\in\C\}$ of $\mathcal{F}_a$, we obtain
$$T_x(F_a^{-1}(F_a(x)))=\mathrm{span}\{[x,df(x+\lambda a)^{\vee}]\colon f\in \C[\g]^G, \lambda\in \C\}.$$
Now note that $[x+\lambda a, df(x+\lambda a)^\vee] =0$ for all $f\in\C[\g]^G$ and $\lambda\in\mathbb{C}$ (see \cite[Proposition 1.3]{KostantSolution}), which together with the last sentence yields \eqref{eq:TangentspaceFibre3}.
\end{proof}

\section{Irreducible components of Mishchenko--Fomenko fibres}\label{Section: Irreducible components}

Our attention now turns to the irreducible components of Mishchenko--Fomenko fibres. Some of our results specialize to those of Charbonnel--Moreau \cite{Charbonnel-Moreau} when $a\in\mathfrak{g}_{\text{reg}}$ is taken to be nilpotent. This reflects the degree to which \cite{Charbonnel-Moreau} inspired our work. 

\subsection{Irreducible components contained in parabolics}\label{Section:IrrCompPar}\label{Subsection: Parabolic components}
Suppose that a parabolic subalgebra $\p\subseteq\g$ contains $a\in\mathfrak{g}_{\text{reg}}$. Proposition \ref{Prop:alregular} then explains how to choose a Levi factor $\lf\subseteq\p$ such that $a_{\lf}\in\lf_{\text{reg}}$. We could next choose free generators of the invariant polynomials on $\lf$, subsequently using such generators to construct a Mishchenko--Fomenko map $F_{a_{\lf}}:\lf\rightarrow\mathbb{C}^{b(\lf)}$ (cf. Section \ref{Subsection: The Mishchenko-Fomenko subalgebra}). By Remark \ref{Remark: Independence of generators}, the fibres $F_{a_{\lf}}^{-1}(F_{a_{\lf}}(x))$ are independent of the aforementioned generators. It is with this understanding that we formulate the following result.     

\begin{thm}\label{Theorem: IrredCompInParabolic_x}
Let $a\in\mathfrak{g}_{\emph{reg}}$ be contained in a parabolic subalgebra $\p\subseteq\g$ having nilpotent radical $\uu$. Let $\h\subseteq\p$ be a Cartan subalgebra containing the semisimple part of $a$, and let $\lf\subseteq\p$ be the $\h$-stable Levi factor.
Consider the Mishchenko--Fomenko map $F_{a_{\lf}}:
\lf\to\mathbb{C}^{b(\lf)}$, and suppose that we have a point $x = x_{\lf} + x_{\uu}\in\p$ with $x_{\lf}\in\lf$ and $x_{\uu}\in\uu$. If $Y$ is an irreducible component of $F_{a_{\lf}}^{-1}(F_{a_{\lf}}(x_{\lf}))$ containing $x_{\lf}$, then $Y+\uu$ is an irreducible component of $F_a^{-1}(F_a(x))$ containing $x$. We thereby obtain a bijection 
\begin{eqnarray*}
\{\text{irred. comp. $Y\subseteq F_{a_{\mathfrak{l}}}^{-1}(F_{a_{\mathfrak{l}}}(x_{\mathfrak{l}}))$ s.t. $x_{\mathfrak{l}}\in Y$}\} &\to& \{\text{irred. comp. $Z\subseteq F_{a}^{-1}(F_{a}(x))$ s.t. $x\in Z\subseteq \mathfrak{p}$}\} \\
Y &\mapsto & Y+\uu.
\end{eqnarray*}
\end{thm}
\begin{proof}
We begin by claiming that $F_a$ is constant-valued on $F_{a_{\lf}}^{-1}(F_{a_{\lf}}(x_{\lf}))$. To this end, suppose that $y\in F_{a_{\lf}}^{-1}(F_{a_{\lf}}(x_{\lf}))$. An application of Lemma \ref{Lemma:3.1.43.parabolic} establishes that
$$f(y+\lambda a) = f(y+\lambda a_{\mathfrak l})$$ for all $f\in\mathbb{C}[\g]^G$ and $\lambda\in\mathbb{C}$, where $a_{\lf}$ is the projection of $a\in\p$ onto $\lf$. Now observe that Remark \ref{Remark: Independence of generators} applies to $F_{a_{\lf}}$ and shows that the right hand side is $f(x_{\lf}+\lambda a_{\lf})$, i.e. 
$$f(y+\lambda a)=f(x_{\lf}+\lambda a_{\lf}).$$ A second application of Lemma \ref{Lemma:3.1.43.parabolic} then gives 
$$f(y+\lambda a)=f(x_{\lf}+\lambda a).$$ Since this holds for all $f\in\mathbb{C}[\g]^G$ and $\lambda\in\mathbb{C}$, Remark \ref{Remark: Independence of generators} implies that $y\in F_a^{-1}(F_a(x_{\lf}))$. We conclude that $F_a$ is indeed constant-valued on $F_{a_{\lf}}^{-1}(F_{a_{\lf}}(x_{\lf}))$.

Now let $Y$ be an irreducible component of $F_{a_{\mathfrak{l}}}^{-1}(F_{a_{\mathfrak{l}}}(x_{\mathfrak{l}}))$ such that $x_{\lf}\in Y$, noting that $F_a$ must also be constant-valued on $Y$. An application of Lemma \ref{Lemma:3.1.43.parabolic} then shows that $F_a$ is constant-valued on $Y+\uu$. Since $x=x_{\lf}+x_{\uu}\in Y+\uu$, this means precisely that $Y+\uu\subseteq F_a^{-1}(F_a(x))$.
At the same time, Proposition \ref{Proposition: Irreducible component dimension} implies that
$$\dim(Y+\uu)=b(\lf)-\mathrm{rk}(\lf)+\dim(\uu).$$
We also know that $b(\mathfrak{g}) = b(\mathfrak{l}) + \dim(\uu)$ and $\mathrm{rk}(\mathfrak{g}) = \mathrm{rk}(\mathfrak{l})$, meaning that 
$$\dim(Y+\uu)=b(\g)-\mathrm{rk}(\g).$$
By Proposition \ref{Proposition: Irreducible component dimension}, the right hand side is precisely the dimension of each irreducible component in $F_a^{-1}(F_a(x))$. The irreducibility of $Y+\uu$ now implies that $Y+\uu$ must be an irreducible component of $F_a^{-1}(F_a(x))$. Having already noted that $x\in Y+\uu$, we have verified the first part of our proposition. 

It remains only to prove that $Y\mapsto Y+\uu$ defines a bijection between the above-indicated sets of irreducible components. This association is clearly injective, reducing us to verifying surjectivity. To this end, let $Z\subseteq\mathfrak{p}$ be an irreducible component of $F_a^{-1}(F_a(x))$ such that $x\in Z\subseteq\p$. The decomposition $\p=\lf\oplus\uu$ induces a projection $\p\rightarrow\lf$, and we let $Z_{\lf}\subseteq\lf$ denote the image of $Z$ under this projection. Lemma \ref{Lemma:3.1.43.parabolic} then implies that $Z_{\lf}+\uu\subseteq F_a^{-1}(F_a(x))$. At the same time, it is easy to see that $Z_{\mathfrak{l}}+\uu$ is irreducible and contains $Z$. These last two sentences force $Z=Z_{\lf}+\uu$ to hold. 

In light of the previous paragraph, it suffices to show that $Z_{\lf}$ is an irreducible component of $F_{a_{\lf}}^{-1}(F_{a_{\lf}}(x_{\lf}))$ containing $x_{\lf}$. We first note that the containment $x_{\lf}\in Z_{\lf}$ is clear, as $x\in Z$. To show that $Z_{\lf}$ is an irreducible component of $F_{a_{\lf}}^{-1}(F_{a_{\lf}}(x_{\lf}))$, we study the intersection $F_{a}^{-1}(F_a(x))\cap\lf$. Let $X$ be an irreducible component of $F_{a}^{-1}(F_a(x))\cap\lf$, noting that $X+\uu$ is irreducible. Lemma \ref{Lemma:3.1.43.parabolic} also implies that $X+\uu\subseteq\ F_a^{-1}(F_a(x))$, so that 
Proposition \ref{Proposition: Irreducible component dimension} yields
$$\dim(X+\uu)\leq b(\g)-\mathrm{rk}(\g).$$
Since $b(\lf)=b(\g)-\dim(\uu)$ and $\mathrm{rk}(\lf)=\mathrm{rk}(\g)$, this amounts to the statement that
$$\dim(X)\leq b(\lf)-\mathrm{rk}(l).$$ It follows that each irreducible component of $F_a^{-1}(F_a(x))\cap\lf$ has dimension at most $b(\lf)-\mathrm{rk}(l)$. In particular, any closed, irreducible, ($b(\lf)-\mathrm{rk}(l)$)-dimensional subvariety of $F_a^{-1}(F_a(x))\cap\lf$ is necessarily an irreducible component of $F_a^{-1}(F_a(x))\cap\lf$.

Now recall the first paragraph of this proof. One can use similar ideas to show that the component functions of $F_a\big\vert_{\lf}:\lf\rightarrow\mathbb{C}^b$ belong to the subalgebra $\mathcal{F}_{a_{\lf}}\subseteq\mathbb{C}[\lf]$. This implies that $F_{a}^{-1}(F_a(x))\cap\lf$ is a union of fibres of $F_{a_{\lf}}$. We may therefore write
$$F_{a}^{-1}(F_a(x))\cap\lf=\bigcup_{z\in S} F_{a_{\lf}}^{-1}(z)$$ for some subset $S\subseteq\mathbb{C}^{b(\lf)}$. For each $z\in S$, let $X_{z,1},\ldots,X_{z,n(z)}$ be the irreducible components of $F_{a_{\lf}}^{-1}(z)$. We then have
\begin{equation}\label{Equation: Irreducible component decomposition} F_{a}^{-1}(F_a(x))\cap\lf=\bigcup_{z\in S}\bigcup_{i=1}^{n(z)}X_{z,i},\end{equation}
while Proposition \ref{Proposition: Irreducible component dimension} implies that $\dim(X_{z,i})=b(\lf)-\mathrm{rk}(\lf)$ for all $z\in S$ and $i\in\{1,\ldots,n(z)\}$. Together with the last sentence of the previous paragraph, this implies that \eqref{Equation: Irreducible component decomposition} is the decomposition of $F_{a}^{-1}(F_a(x))\cap\lf$ into irreducible components.

We now observe that $Z_{\lf}$ is irreducible and satisfies
\begin{align*}
\dim(Z_{\lf}) & = \dim(Z_{\lf}+\uu)-\dim(\uu) \\
&=\dim(Z)-\dim(\uu)\hspace{52pt}\text{(since $Z=Z_{\lf}+\uu$)}\\
&=b(\g)-\mathrm{rk}(\g)-\dim(\uu) \hspace{33pt}\text{(by Proposition \ref{Proposition: Irreducible component dimension})}\\
&=b(\lf)-\mathrm{rk}(\lf)\hspace{80pt}\text{(since }b(\lf)=b(\g)-\dim(\uu)\text{ and }\mathrm{rk}(\lf)=\mathrm{rk}(\g)\text{)}.
\end{align*}
We may also use Lemma \ref{Lemma:3.1.43.parabolic} and the fact that $Z\subseteq F_{a}^{-1}(F_a(x))$ to establish that $Z_{\lf}\subseteq F_{a}^{-1}(F_a(x))\cap\lf$.
These last two sentences combine with the previous paragraph to imply that $Z_{\lf}=X_{z,i}$ for some $z\in S$ and $i\in\{1,\ldots,n(z)\}$. This means that $Z_{\lf}$ is an irreducible component of a fibre of $F_{a_{\lf}}$. Since $x_{\lf}\in Z_{\lf}$, the fibre in question must be $F_{a_{\lf}}^{-1}(F_{a_{\lf}}(x_{\lf}))$. The proof is complete.
\end{proof}

A similar approach can be used to describe the irreducible components of $F_a^{-1}(F_a(x))$ that are contained in a given parabolic $\mathfrak{p}$, provided that $\p$ contains $a\in\greg$.

\begin{prop}\label{Prop: IrredCompParabolic}
Let $a$, $\p$, $\lf$, and $\uu$ be exactly as described in the first two sentences of Theorem \ref{Theorem: IrredCompInParabolic_x}.  
If $x \in\mathfrak{p}$, then $F_a^{-1}(F_a(x))\cap\mathfrak{l}$ is a union of fibres of $F_{a_{\mathfrak{l}}}$. Furthermore, we have a bijection
{\small \begin{eqnarray*}
\left\{\text{$Y\subseteq F_a^{-1}(F_a(x))\cap\mathfrak{l}$ s.t. $Y$ is an irred. comp. of a fibre of $F_{a_{\mathfrak{l}}}$}\right\} &\to& \left\{\text{irred. comp. $Z\subseteq F_a^{-1}(F_a(x))$ s.t. $Z\subseteq\p$} \right\}\\
Y &\mapsto& Y+\uu.
\end{eqnarray*}}
\end{prop}
\begin{proof}
The proof of Theorem \ref{Theorem: IrredCompInParabolic_x} works with minor modifications. 
\end{proof}

We may build on Theorem \ref{Theorem: IrredCompInParabolic_x} and Proposition \ref{Prop: IrredCompParabolic} as follows.

\begin{prop}\label{Prop: Unipotent orbit}
Suppose that $a\in\greg$ is contained in a Borel subalgebra $\bb\subseteq\g$ with nilpotent radical $\uu$. Let $\mathfrak{h}\subseteq \mathfrak{g}$ be a Cartan subalgebra contained in $\bb$, and assume that $x\in\bb$. 
\begin{itemize}
\item[(i)] The variety $x+\mathfrak{u}$ is the unique irreducible component of $F_a^{-1}(F_a(x))$ that contains $x$ and is contained in $\mathfrak{b}$.
\item[(ii)] If $x$ is regular and semisimple, then $x+\mathfrak{u}$ is an irreducible component of $F_a^{-1}(F_a(x))\cap Gx$.
\item[(iii)] Write $x = x_{\h}+x_{\uu}$ with $x_{\h}\in\h$ and $x_{\uu}\in\uu$, and let $W$ be the Weyl group of $(\g,\h)$. If $a$ is nilpotent, then the irreducible components of $F_a^{-1}(F_a(x))$ contained in $\bb$ are the subvarieties $w x_{\h} +\uu$, $w\in W$.
\end{itemize}
\end{prop}

\begin{proof}
To prove (i), observe that the component functions of $F_{a_{\h}}:\h\rightarrow\mathbb{C}^r$ are homogeneous, algebraically independent generators of $\mathbb{C}[\h]$. It follows that $F_{a_{\h}}$ is a vector space isomorphism, so that $F_{a_{\h}}^{-1}(F_{a_{\h}}(x_{\h}))=\{x_{\h}\}$.
With this in mind, consider the bijective correspondence implied by Theorem \ref{Theorem: IrredCompInParabolic_x} when $\mathfrak{p} = \mathfrak{b}$ and $\mathfrak{l} = \mathfrak{h}$. We conclude that $x_{\h}+\uu=x+\mathfrak{u}$ is the unique irreducible component of $F_a^{-1}(F_a(x))$ that contains $x$ and is contained in $\mathfrak{b}$.

To prove (ii), we recall from Lemma \ref{Lemma: Contained in closure} that $F_a^{-1}(F_a(x))\subseteq\overline{Gx}$. The orbit $Gx$ is open in its closure, which together with the previous sentence means that $F_a^{-1}(F_a(x))\cap Gx$ is open in $F_a^{-1}(F_a(x))$. Since $x+\mathfrak{u}$ is an irreducible component of $F_a^{-1}(F_a(x))$ that intersects this open set, it follows that $(x+\mathfrak{u})\cap Gx$ is an irreducible component of $F_a^{-1}(F_a(x))\cap Gx$. We are therefore reduced to proving that $(x+\mathfrak{u})\cap Gx=x+\mathfrak{u}$. Accordingly, let $B$ denote the Borel subgroup of $G$ having Lie algebra $\mathfrak{b}$. The result \cite[Lemma 3.1.44]{Chriss} then implies than $Bx=x+\mathfrak{u}$, so that $(x+\mathfrak{u})\cap Gx=x+\mathfrak{u}$. This completes the proof.  

To prove (iii), recall that (i) gives $x+\uu\subseteq F_a^{-1}(F_a(x))$. It follows that $x_{\h}\in F_a^{-1}(F_a(x))$, so that $F_a^{-1}(F_a(x)) = F_a^{-1}(F_a(x_{\h}))$. We also note that $a\in\uu$, a consequence of $a$ being nilpotent. A straightforward application of Lemma \ref{Lemma:3.1.43.parabolic} then shows that \begin{equation}\label{Equation: Nilpotent equation} F_a\big\vert_{\h} = (f_1\big\vert_{\h},\ldots,f_r\big\vert_{\h},0,\ldots,0): \h\to \C^b.\end{equation} We thus have $$F_a^{-1}(F_a(x))\cap\h = F_a^{-1}(F_a(x_{\h}))\cap\h = F^{-1}(F(x_{\h}))\cap\h,$$ where $F:=(f_1,\ldots,f_r):\g\rightarrow\mathbb{C}^r$. At the same time, \cite[Lemma 9.2]{KostantTDS} tells us that $F^{-1}(F(x_{\h}))\cap\h$ is the $W$-orbit of $x_{\h}$, i.e.
$$F_a^{-1}(F_a(x))\cap\h=\{wx_{\h}:w\in W\}.$$
The desired result now follows from setting $\p=\bb$ and $\lf=\h$ in Proposition \ref{Prop: IrredCompParabolic}, along with the observation that $F_{a_{\h}}:\h\rightarrow\mathbb{C}^r$ has only singleton fibres (see the proof of (i)).  
\end{proof}

\begin{rem}
It is illuminating to consider Proposition \ref{Prop: Unipotent orbit}(ii) in the case of a regular element $x\in\bb$. In this case, the last $b-r$ components of $F_a$ form a completely integrable system on the symplectic leaf $Gx\subseteq\g$ (see \cite{CRR} or \cite{Panyushev2}). Proposition \ref{Prop: Unipotent orbit}(ii) then amounts to the following statement: $x+\uu$ is an irreducible component of the fibre through $x$ in the integrable system on $Gx$.    
\end{rem}

\begin{rem}
The irreducible component $x+\mathfrak{u}\subseteq F_a^{-1}(F_a(x))$ constructed in Proposition \ref{Prop: Unipotent orbit}(i) need not meet $\gsreg$. To see this, let $a$ be any regular nilpotent element, $\mathfrak{b}$ the unique Borel subalgebra of $\mathfrak{g}$ containing $a$, and $\mathfrak{u}$ the nilpotent radical of $\mathfrak{b}$. It then follows that $a\in\mathfrak{u}$. Now choose a Cartan subalgebra $\mathfrak{h}$ of $\mathfrak{g}$ contained in $\mathfrak{b}$, which then renders $\mathfrak{u}$ a sum of the positive root spaces. We may therefore write
\[
a=\sum_{\alpha\in\Delta_{+}}r_{\alpha},
\]
where $\Delta_{+}$ is the set of positive roots and $r_{\alpha}\in\mathfrak{g}_{\alpha}$ for all $\alpha\in\Delta_{+}$. Since $a$ is regular, it follows from \cite[Theorem 5.3]{KostantTDS} that $r_{\alpha}\neq 0$ for all simple roots $\alpha$.

Now take an arbitrary $y\in\mathfrak{u}$, writing 
\[
y=\sum_{\alpha\in\Delta_{+}}s_{\alpha}
\]
with $s_{\alpha}\in\mathfrak{g}_{\alpha}$ for all $\alpha\in\Delta_{+}$. If we fix a particular simple root $\alpha$, then $r_{\alpha}$ being non-zero allows us to find $\lambda\in\mathbb{C}$ with $s_{\alpha}+\lambda r_{\alpha}=0$. We conclude that $y+\lambda a\not\in\mathfrak{g}_{\text{reg}}$, by \cite[Theorem 5.3]{KostantTDS}. It follows that $y\not\in\gsreg$, implying that $\mathfrak{u}\cap\gsreg=\emptyset$. In particular, every $x\in\mathfrak{u}$ has the property that $x+\mathfrak{u}$ ($=\mathfrak{u}$) does not meet $\gsreg$.
\end{rem}

\subsection{A recursive formula}\label{Subsection: A recursive formula}
We now use the results established in Section \ref{Section:IrrCompPar} to derive a recursive formula for the number of irreducible components in $F_a^{-1}(0)$. To this end, the following elementary lemma is needed. 

\begin{lem}\label{Lemma:zero_in_irrcomp}
If $a\in\greg$, then every irreducible component of $F_a^{-1}(0)$ contains the origin $0\in\g$. 
\end{lem}
\begin{proof}
The formula \eqref{eq:Def_fij} implies that each polynomial $f_{ij}^a\in\mathbb{C}[\g]$ is homogeneous of degree $d_i-j$. It follows that $F_a^{-1}(0)$ is invariant under the dilation action of $\mathbb{C}^{\times}$ on $\mathfrak{g}$. This forces each irreducible component of $F_a^{-1}(0)$ to be invariant under the aforementioned $\mathbb{C}^{\times}$-action. Since each of these irreducible components is also closed, it must contain $0$.
\end{proof}

We now introduce the notation on which our recursive formula is based. Given $a\in\greg$, let us set 
\[
\mathcal B_a :=\{\text{Borel subalgebras $\bb$ of $\g$ s.t. $a\in\bb$}\},
\] 
\[
\widetilde{\mathcal P_a} :=\{\text{parabolic subalgebras $\p$ of $\g$ s.t. $\p\neq\g$ and $a\in\p$}\},
\]
and 
\[
\mathcal P_a := \widetilde{\mathcal P_a}\setminus \mathcal B_a.
\]
We also consider the sets
\[
^{\g}\mathcal I_a := \{\text{irreducible components of $F_a^{-1}(0)$}\}
\]
and 
\[
^{\g}\mathcal I_a' := \{Z\in{}^{\g}\mathcal I_a\colon Z\nsubseteq \mathfrak p \; \forall \mathfrak p\in\mathcal P_a\}.
\]

\begin{rem}
One can phrase the definition of $^{\g}\mathcal I_a'$ in slightly different terms. To see this, note that the irreducible components of $\bigcup_{\p\in\mathcal{P}_a}\p$ are the maximal elements of $\mathcal{P}_a$ with respect to inclusion. It follows that a closed, irreducible subvariety of $\g$ belongs to $\bigcup_{\p\in\mathcal{P}_a}\p$ if and only it belongs to some $\p\in\mathcal{P}_a$. We deduce that $$^{\g}\mathcal I_a' = \{Z\in{}^{\g}\mathcal I_a\colon Z\nsubseteq \bigcup_{\p\in\mathcal P_a}\mathfrak p\}.$$
\end{rem}

Given any fixed $\mathfrak{p}\in\mathcal{P}_a$, we write
\[
^{\g}\mathcal I_a^{\mathfrak p} := \{Z\in{}^{\g}\mathcal I_a\colon \text{$Z\subseteq\mathfrak p$ and $Z\nsubseteq \mathfrak p'$ for all $\mathfrak p'\in \mathcal P_a$ with $\mathfrak p'\subsetneq \mathfrak p$}\}.
\]

Let us consider the Jordan decompositon $a=s+n$ and choose a Cartan subalgebra $\h\subseteq\g$ such that $s\in\h$. Given any $\p\in\mathcal{P}_a$, let $\lf_{\p}$ denote the unique $\h$-stable Levi factor and $a_{\lf_{\p}}\in\lf_{\p}$ the projection of $a$ onto $\lf_{\p}$. Proposition \ref{Prop:alregular} tells us that $a_{\lf_{\p}}$ is regular in $\lf_{\p}$, and we thus have analogous definitions of $\mathcal{B}_{a_{\mathfrak{l_{\p}}}}$, $\widetilde{\mathcal{P}_{a_\mathfrak{l_{\p}}}}$, $\mathcal{P}_{a_\mathfrak{l_{\p}}}$, $^{\lf_{\p}}\mathcal I_{a_{\lf_{\p}}}$, and $^{\lf_{\p}}\mathcal I_{a_{\lf_{\p}}}'$. Now decompose $[\mathfrak{l}_{\mathfrak{p}},\mathfrak{l}_{\p}]$ into its simple factors $\mathfrak{l}_1,\ldots,\mathfrak{l}_N$, so that
\[
\mathfrak{l}_{\p}=\mathfrak{z}(\mathfrak{l}_{\p})\oplus\mathfrak{l}_1\oplus\dots\oplus\mathfrak{l}_N.
\]
If $a_{\lf_i}$ denotes the projection of $a_{\lf_{\p}}$ onto $\lf_{i}$, then $a_{\lf_i}\in(\lf_i)_{\text{reg}}$ and one has analogous definitions of $\mathcal{B}_{a_{\mathfrak{l}_i}}$, $\widetilde{\mathcal{P}_{a_{\mathfrak{l}_i}}}$, $\mathcal{P}_{a_{\mathfrak{l}_i}}$, $^{\lf_i}\mathcal I_{a_{\lf_i}}$, and $^{\lf_i}\mathcal I_{a_{\lf_i}}'$.
\begin{thm}\label{Thm:RecursiveFormula}
Using the notation explained above, we have the following recursive formula for $|^{\g}\mathcal I_a|$:
\begin{equation}\label{Equation:RecursiveFormula}
|^{\g}\mathcal I_a| = |^{\g}\mathcal I_a'| + \sum_{\mathfrak{p}\in\mathcal P_a}\left(\prod_{\text{$\mathfrak l\subseteq [\mathfrak l_{\p},\mathfrak l_{\p}]$\\ \hspace{1pt} \emph{simple factor}}}|^{\mathfrak l}I_{a_{\mathfrak l}}'|\right) + |\mathcal B_a|.
\end{equation}
\end{thm}
\begin{proof}
Proposition \ref{Prop: Unipotent orbit} implies that each $\mathfrak b\in \mathcal B_a$ yields an irreducible component $[\bb,\bb]$ of $F_a^{-1}(0)$. At the same time, observe that $^{\g}\mathcal I_a^{\mathfrak p}\cap \; ^{\g}\mathcal I_a^{\mathfrak p'} = \emptyset$ if $\mathfrak p \neq \mathfrak p'$. These last two sentences allow us to write $^{\g}\mathcal I_a$ as the disjoint union 
\[
^{\g}\mathcal I_a  = \; ^{\g}\mathcal I_a' \cup \left(\bigcup_{\mathfrak p\in\mathcal P_a} {^{\g}\mathcal I_a^{\mathfrak p}}\right) \cup \{[\bb,\bb]\colon \bb\in\mathcal B_a\}.
\]
It therefore suffices to prove that 
\[
|^{\g}\mathcal I_a^{\mathfrak p}| = \prod_{\text{$\mathfrak l\subseteq [\mathfrak l_{\p},\mathfrak l_{\p}]$\\ \hspace{1pt} \text{simple factor}}}|^{\mathfrak l}\mathcal I_{a_{\mathfrak l}}'|
\]
for each $\p\in\mathcal{P}_a$. Our approach is similar to that appearing in the proof of \cite[Proposition 52]{Charbonnel-Moreau}, and the relevant details are given below. 

Fix an element $\p\in\mathcal{P}_a$ and let $\mathfrak{u}_{\p}$ denote its nilpotent radical. Lemma \ref{Lemma:zero_in_irrcomp} and Theorem \ref{Theorem: IrredCompInParabolic_x} tell us that $Y\mapsto Y+\mathfrak{u}_{\mathfrak{p}}$ defines a bijection from elements of $^{\mathfrak{l}_{\mathfrak p}}\mathcal I_{a_{\mathfrak{l}_{\mathfrak{p}}}}$ to those elements of $^{\g}\mathcal I_a$ that are contained in $\p$. This restricts to a bijection between $^{\mathfrak{l}_{\mathfrak p}}\mathcal I_{a_{\mathfrak{l}_{\mathfrak{p}}}}'$ and $^{\g}\mathcal I_a^{\p}$. We are therefore reduced to proving that
\[
|^{\mathfrak{l}_{\mathfrak p}}\mathcal I_{a_{\mathfrak{l}_{\mathfrak{p}}}}'| = \prod_{\text{$\mathfrak l\subseteq [\mathfrak l_{\p},\mathfrak l_{\p}]$\\ \hspace{1pt} \text{simple factor}}}|^{\mathfrak l}\mathcal I_{a_{\mathfrak l}}'|.
\]

Let us decompose $[\mathfrak{l}_{\mathfrak{p}},\mathfrak{l}_{\p}]$ into its simple factors $\mathfrak{l}_1,\ldots,\mathfrak{l}_N$. It follows that the elements of $^{\mathfrak{l}_{\mathfrak p}}\mathcal I_{a_{\mathfrak{l}_{\mathfrak{p}}}}$ are precisely those varieties of the form
$$\{0\}\times Z_1\times\dots\times Z_N\subseteq\mathfrak{z}(\mathfrak{l}_{\p})\oplus\mathfrak{l}_1\oplus\dots\oplus\mathfrak{l}_N=\mathfrak{l}_{\p},$$
where each $Z_i$ is an irreducible component of $F_{a_{\mathfrak{l}_i}}^{-1}(0)\subset \mathfrak{l}_i$. We therefore have bijections 
\[
^{\mathfrak{l}_{\mathfrak p}}\mathcal I_{a_{\mathfrak{l}_{\mathfrak{p}}}} \cong \;^{\mathfrak{l}_{1}}\mathcal I_{a_{\mathfrak{l}_{1}}}\times \dots \times\; ^{\mathfrak{l}_{N}}\mathcal I_{a_{\mathfrak{l}_{N}}}
\] 
and 
\[
^{\mathfrak{l}_{\mathfrak p}}\mathcal I_{a_{\mathfrak{l}_{\mathfrak{p}}}}' \cong \;^{\mathfrak{l}_{1}}\mathcal I_{a_{\mathfrak{l}_{1}}}'\times \dots \times\; ^{\mathfrak{l}_{N}}\mathcal I_{a_{\mathfrak{l}_{N}}}',
\] 
proving the theorem.
\end{proof}

\begin{rem}
Lemma 51 in \cite{Charbonnel-Moreau} shows that $^{\g}\mathcal I_a'\neq \emptyset$ if $a\in\greg$ is nilpotent. In Secton \ref{Sec:Examples}, we will see that $^{\g}\mathcal I_a'\neq \emptyset$ for every $a\in\g_{\text{reg}}$ when $\g=\mathfrak{sl}_3(\mathbb{C})$.  
\end{rem}

\begin{rem}
If $a\in\greg$ is nilpotent, then $|\mathcal B_a| =1$ and Theorem \ref{Thm:RecursiveFormula} reduces to \cite[Proposition 52]{Charbonnel-Moreau}. If $a\in\greg$ is semisimple, then $|\mathcal B_a| = |W|$ in the recursive formula \eqref{Equation:RecursiveFormula}.
\end{rem}

\begin{rem}
Our recursive formula allows us to obtain lower bounds on the numbers $|^{\g} \mathcal{I}_a|$, by determining $|\mathcal B_a|$, $|\mathcal P_a|$ and the decompositions of the Levi factors into simple parts. If one works in Type $A_r$ for small $r$, then these numbers and Levi factors are easy to obtain. We refer the reader to Section \ref{Sec:Examples} for further details.
\end{rem}

\subsection{Exotic irreducible components}  \label{Subsection: Exotic Components}
Fix $a\in\greg$ and recall the statements of Theorem \ref{Theorem: IrredCompInParabolic_x} and Proposition \ref{Prop: IrredCompParabolic}. Note that these become completely tautological in the case $\p=\g$, so that interesting results necessitate taking $\p\neq\g$. In this latter case, every irreducible component constructed via Theorem \ref{Theorem: IrredCompInParabolic_x} and Proposition \ref{Prop: IrredCompParabolic} is constrained to lie in some $\p\in \widetilde{\mathcal P_a}$. It is therefore natural to consider irreducible components of Mishchenko--Fomenko fibres that do not lie in a proper parabolic subalgebra containing $a$. We begin with the following result.    

\begin{prop}\label{Proposition: Tarasov_vs_parabolic}
Suppose that $a\in\mathfrak{g}_{\emph{reg}}$ is semisimple. If $\mathfrak{p}\in\widetilde{\mathcal P_a}$, then every fibre of $F_a$ has an irreducible component that is not contained in $\mathfrak{p}$.
\end{prop}

\begin{proof}
Observe that $\mathfrak{h}:=\mathfrak{g}_a$ is a Cartan subalgebra of $\mathfrak{g}$ contained in $\mathfrak{p}$. Let $\Delta\subseteq\mathfrak{h}^*$ be the associated set of roots, noting that
\begin{equation}\label{Equation: Parabolic decomposition}\mathfrak{p}=\mathfrak{h}\oplus\bigoplus_{\alpha\in Q}\mathfrak{g}_{\alpha}\end{equation} for some proper subset $Q\subsetneq\Delta$. Choose an element $\beta\in\Delta\setminus Q$, as well as a set of negative roots $\Delta_{-}\subseteq\Delta$ with $\beta\in\Delta_{-}$. It follows that $\Delta_{+}:=-\Delta_{-}$ is the associated choice of positive roots, and that we have the opposite Borel subalgebras
$$\mathfrak{b}:=\mathfrak{h}\oplus\bigoplus_{\alpha\in\Delta_{+}}\mathfrak{g}_{\alpha}\quad\text{and}\quad\mathfrak{b}_{-}:=\mathfrak{h}\oplus\bigoplus_{\alpha\in\Delta_{-}}\mathfrak{g}_{\alpha}.$$ 

Now let $\Pi\subseteq\Delta_{+}$ be the set of simple roots and consider \begin{equation}\label{Equation: Nilpotent element}\xi:=\sum_{\alpha\in\Pi}e_{-\alpha},\end{equation} where $e_{-\alpha}\in\mathfrak{g}_{-\alpha}\setminus\{0\}$ for each $\alpha\in\Pi$. The subset $\xi+\mathfrak{b}\subseteq\mathfrak{g}$ is then a section of $F_a$ (see Theorem \ref{Theorem: Section}), so that it will suffice to prove $(\xi+\mathfrak{b})\cap\mathfrak{p}=\emptyset$. 

Assume that $(\xi+\mathfrak{b})\cap\mathfrak{p}\neq\emptyset$. It follows that $\xi+x\in\mathfrak{p}$ for some $x\in\mathfrak{b}$, which by \eqref{Equation: Parabolic decomposition} and \eqref{Equation: Nilpotent element} implies that $-\Pi\subseteq Q$. The subalgebra $\mathfrak{p}$ then contains $\mathfrak{h}$ and every negative simple root space, forcing $\mathfrak{b}_{-}\subseteq\mathfrak{p}$ to hold. At the same time, our condition $\beta\in\Delta_{-}$ yields the inclusion $\mathfrak{g}_{\beta}\subseteq\mathfrak{b}_{-}$. It follows that $\mathfrak{g}_{\beta}\subseteq\mathfrak{p}$, contradicting the fact that $\beta\not\in Q$. We conclude that $(\xi+\mathfrak{b})\cap\mathfrak{p}=\emptyset$, completing the proof. 
\end{proof}

While Proposition \ref{Proposition: Tarasov_vs_parabolic} considers irreducible components that are not contained in a fixed $\p\in\widetilde{\mathcal P_a}$, one could ask about components not contained in any $\p\in\widetilde{\mathcal P_a}$. We formalize the latter situation as follows. 

\begin{defn}
Suppose that $a\in\greg$ and $x\in\g$. We call an irreducible component $Z\subseteq F_a^{-1}(F_a(x))$ \textit{exotic} if $Z\nsubseteq\p$ for all $\p\in\widetilde{\mathcal{P}_a}$.
\end{defn}

Our next result identifies a family of fibres that have exotic irreducible components.  

\begin{prop}\label{Proposition: Not contained}
Assume that $\g$ is simple, and let $a\in\g_{\emph{reg}}$ be a semisimple element. Choose a collection of simple positive roots with respect to the Cartan subalgebra $\h:=\g_a$, and denote the resulting positive Borel subalgebra by $\mathfrak{b}\subseteq\g$. Let $\xi\in\g$ be a sum of non-zero negative simple root vectors, one for each negative simple root. If $x\in\xi+\mathfrak{b}$ has a non-zero component in the highest root space, then the fibre $F_a^{-1}(F_a(x))$ has an exotic irreducible component. 
\end{prop}

\begin{proof}
Denote the sets of roots, positive roots, and simple roots by $\Delta$, $\Delta_{+}$, and $\Pi$, respectively, so that
\begin{equation}\label{Equation: Positive Borel}\mathfrak{b}=\mathfrak{h}\oplus\bigoplus_{\alpha\in\Delta_{+}}\mathfrak{g}_{\alpha}.\end{equation}
Let us also set
$$V:=\mathfrak{h}\oplus\bigoplus_{\alpha\in\Delta_{+}\setminus\{\lambda\}}\mathfrak{g}_{\alpha},$$ where $\lambda\in\Delta_{+}$ is the highest root. Since $\xi+\bb$ is a section of $F_a$ (see Theorem \ref{Theorem: Section}), it suffices to prove that $(\xi+\mathfrak{b})\setminus(\xi+V)$ is disjoint from $\bigcup_{\p\in\widetilde{\mathcal{P}_a}}\p$. This is equivalent to verifying that $(\xi+\mathfrak{b})\cap\p\subseteq\xi+V$ for all $\p\in\widetilde{\mathcal{P}_a}$. 

Let $\mathfrak{p}\in\widetilde{\mathcal{P}_a}$ be given, noting that $\mathfrak{h}\subseteq\mathfrak{p}$. It follows that \begin{equation}\label{Equation: New parabolic decomposition} \mathfrak{p}=\mathfrak{h}\oplus\bigoplus_{\alpha\in\Delta_{\mathfrak{p}}}\mathfrak{g}_{\alpha}\end{equation} for some proper subset $\Delta_{\mathfrak{p}}\subsetneq\Delta$. We then have the following two possibilities: $-\Pi\nsubseteq\Delta_{\mathfrak{p}}$ or $-\Pi\subseteq\Delta_{\mathfrak{p}}$. In the first case, \eqref{Equation: Positive Borel}, \eqref{Equation: New parabolic decomposition}, and the definition of $\xi$ collectively force $\big(\xi+\mathfrak{b}\big)\cap\mathfrak{p}=\emptyset$ to hold. This certainly implies that $\big(\xi+\mathfrak{b}\big)\cap\mathfrak{p}\subseteq\xi+V$, as desired. 

If $-\Pi\subseteq\Delta_{\mathfrak{p}}$, then one has $$\big(\xi+\mathfrak{b}\big)\cap\mathfrak{p}=\xi+\bigg(\mathfrak{h}\oplus\bigoplus_{\alpha\in\Delta_{\mathfrak{p}}\cap\Delta_{+}}\mathfrak{g}_{\alpha}\bigg).$$ We conclude that $\big(\xi+\mathfrak{b}\big)\cap\mathfrak{p}\subseteq\xi+V$ if and only if $\lambda\not\in\Delta_{\mathfrak{p}}$. The latter condition is best investigated via the following classical fact: $\mathfrak{g}_{\lambda}$ and the negative simple root spaces generate $\mathfrak{g}$. Since $\mathfrak{p}$ is properly contained in $\mathfrak{g}$ and contains the negative simple root spaces, we must have $\lambda\not\in\Delta_{\mathfrak{p}}$. It follows that $\big(\xi+\mathfrak{b}\big)\cap\mathfrak{p}\subseteq\xi+V$, completing the proof.   
\end{proof}

\begin{rem}
Since $\xi+\mathfrak{b}$ is a ($b$-dimensional) section of $F_a:\g\rightarrow\mathbb{C}^b$ for all semisimple $a\in\g_{\text{reg}}$, Proposition \ref{Proposition: Not contained} provides a $b$-dimensional family of fibres with exotic irreducible components. 
\end{rem}

\begin{rem}
With only mild adjustments, one can formulate Proposition \ref{Proposition: Not contained} for semisimple Lie algebras $\mathfrak{g}$.
\end{rem}

\section{Singularities in Mishchenko--Fomenko fibres}\label{Section: Singularities}
Our attention now turns to the smooth and singular loci of Mishchenko--Fomenko fibres. Section \ref{Subsection: Critical values} studies the critical values of $F_a$, while Section \ref{Subsection: A family of singular fibres} elucidates a role for the subalgebra $\ba\subseteq\g$ discussed in Section \ref{subsec:ba}.  

\subsection{Critical values}\label{Subsection: Critical values}

Recall that the critical points of $F_a$ constitute the set $\Sing^a=\gsing+\mathbb{C}a\subseteq\g$ (see Section \ref{Subsection: The Mishchenko-Fomenko subalgebra}). Our objective is to discuss the set of critical values $F_a(\Sing^a)\subseteq\mathbb{C}^b$, and this necessitates using the following rephrased version of \cite[Theorem 4.2]{Veldkamp}. 

\begin{thm}\label{Theorem: Veldkamp}
The codimension of $\gsing$ in $\mathfrak{g}$ is $3$.
\end{thm}

We may now gauge the codimension of the closure $\overline{F_a(\Sing^a)}\subseteq\mathbb{C}^b$.  

\begin{prop}\label{Proposition:dimF(Sing)}
If $a\in\mathfrak{g}_{\emph{reg}}$, then $\overline{F_a(\Sing^a)}\subseteq\mathbb{C}^b$ has codimension $1$ or $2$ in $\mathbb{C}^b$.
\end{prop}
\begin{proof}
By Theorem \ref{Theorem: Veldkamp}, some irreducible component $X$ of $\Sing^a=\mathfrak{g}_{\text{sing}}+\mathbb{C}a$ has codimension $2$ in $\mathfrak{g}$. Let $j:X\rightarrow\mathfrak{g}$ be the inclusion map, observing that $j^*:\mathbb{C}[\mathfrak{g}]\rightarrow\mathbb{C}[X]$ is surjective with kernel equal to a prime ideal $I\subseteq\mathbb{C}[\mathfrak{g}]$. 

We claim that the prime ideal $I\cap\mathcal{F}_a\subseteq\mathcal{F}_a$ has height at most $2$. To see this, we first note that $\mathcal{F}_a\subseteq\mathbb{C}[\g]$ is a flat extension of rings (see \cite[Theorem 1.2]{Moreau}). It follows from \cite[Theorem 9.5]{Matsumura} that this extension satisfies the \textit{going-down property} for prime ideals. Now let $J_0\subsetneq J_1\subsetneq\cdots\subsetneq J_n=I\cap\mathcal{F}_a$ be a strictly increasing sequence of prime ideals in $\mathcal{F}_a$. The going-down property allows us to find a strictly increasing sequence $I_0\subsetneq I_1\subsetneq\cdots\subsetneq I_n=I$ of prime ideals in $\mathbb{C}[\g]$ satisfying $J_k=I_k\cap\mathcal{F}_a$ for all $k\in\{1,\ldots,n\}$. Note also that $I$ has height $2$ in $\mathbb{C}[\g]$, owing to the fact that $X$ has codimension $2$ in $\g$. It follows that $n\leq 2$, and we conclude that $I\cap\mathcal{F}_a$ has height at most $2$ in $\mathcal{F}_a$.

Now note that $\mathcal{F}_a$ is a polynomial algebra in $b$-many indeterminates (see Theorem \ref{Theorem: Mishchenko-Fomenko}). This combines with the result of the previous paragraph to yield 
\[
\mathrm{Kdim}\bigg((\mathcal{F}_a)/(I\cap\mathcal{F}_a)\bigg)\geq b-2,
\] 
where $\mathrm{Kdim}$ denotes Krull dimension. Observing that $j^*(\mathcal{F}_a)\cong\mathcal{F}_a/(I\cap\mathcal{F}_a)$, we obtain
\[
\mathrm{Kdim}(j^*(\mathcal{F}_a))\geq b-2.
\]
We also know that the functions $j^*(f_1),\ldots,j^*(f_b)\in\mathbb{C}[X]$ generate $j^*(\mathcal{F}_a)$. Taken together, these last two sentences have the following consequence: for all generic $x\in X$, the differentials of $j^*(f_1),\ldots,j^*(f_b)$ at $x$ span a subspace of $T_x^*X$ of dimension at least $b-2$. This is in turn equivalent to the differential of $F_a\big\vert_X:X\rightarrow\mathbb{C}^b$ having rank at least $b-2$ at all generic points of $X$. We conclude that the dimension of $\overline{F_a(X)}\subseteq\mathbb{C}^b$ is at least $b-2$. The inclusion $\overline{F_a(X)}\subseteq \overline{F_a(\Sing^a)}$ then establishes that the latter set has dimension at least $b-2$. On the other hand, it is a straightforward consequence of Sard's theorem that $\overline{F_a(\Sing^a)}$ cannot be $b$-dimensional. It follows that the dimension of $\overline{F_a(\Sing^a)}$ is $b-1$ or $b-2$, completing the proof.
\end{proof}

\begin{rem}
Each of the codimensions $1$ and $2$ is achievable in examples, as we later discuss in Remark \ref{Remark: Possible codimensions}. 
\end{rem}

\subsection{A family of singular fibres}\label{Subsection: A family of singular fibres}
Fix $a\in\greg$ and recall the subalgebra $\ba\subseteq\g$ from Section \ref{subsec:ba}. This subalgebra and its nilpotent radical $\uu^a:=[\ba,\ba]\subseteq\ba$ turn out to play the following role with respect to singular points in Mishchenko--Fomenko fibres.   

\begin{prop}\label{Proposition: Intersecting irreducible components}
Suppose that $a\in\greg$ is not nilpotent. If $x\in\ba$, then $x+\ua$ lies in the singular locus of $F_a^{-1}(F_a(x))$.
\end{prop}

\begin{proof}
Let $a=s+n$ be the Jordan decomposition of $a$, where $s\in\g$ is semisimple and $n\in\g$ is nilpotent. Since $a$ is not nilpotent, $s\neq 0$ and the Levi subalgebra $\mathfrak{g}_s$ is properly contained in $\mathfrak{g}$. It follows that no Borel subalgebra of $\mathfrak{g}_s$ can have the dimension necessary to be a Borel subalgebra of $\mathfrak{g}$. In particular, $\ba$ cannot be a Borel subalgebra of $\mathfrak{g}$. It now follows from Proposition \ref{Proposition: Intersection of Borels} that there exist distinct Borel subalgebras $\mathfrak{b}_1,\mathfrak{b}_2\subseteq\mathfrak{g}$ such that $\ba\subseteq\mathfrak{b}_1\cap\mathfrak{b}_2$. Proposition \ref{Prop: Unipotent orbit} then implies that $x+\mathfrak{u}_1$ and $x+\mathfrak{u}_2$ are irreducible components of $F_a^{-1}(F_a(x))$, where $\mathfrak{u}_1$ and $\mathfrak{u}_2$ are the nilpotent radicals of $\mathfrak{b}_1$ and $\mathfrak{b}_2$, respectively. Since $\bb_1$ and $\bb_2$ are distinct Borel subalgebras, we see that $x+\mathfrak{u}_1$ and $x+\mathfrak{u}_2$ are distinct irreducible components of $F_a^{-1}(F_a(x))$. We also note that $\ua\subseteq\mathfrak{u}_1\cap\mathfrak{u}_2$, implying that $x+\ua$ lies in the intersection of the two components $x+\mathfrak{u}_1$ and $x+\mathfrak{u}_2$. We conclude that $x+\ua$ consists of singular points in $F_a^{-1}(F_a(x))$. 
\end{proof}

One immediate consequence of Proposition \ref{Proposition: Intersecting irreducible components} is that the Mishchenko--Fomenko fibres over $F_a(\ba)\subseteq\mathbb{C}^b$ are singular when $a\in\greg$ is not nilpotent. In other words, the points in $F_a(\ba)$ index a family of singular Mishchenko--Fomenko fibres. This leads us to study $F_a(\ba)$ in more detail, for which the following lemma will be helpful.  

\begin{lem}\label{Lemma: Generate}
Suppose that $a\in\mathfrak{g}_{\emph{reg}}$ has a Jordan decomposition of $a=s+n$, where $s\in\g$ is semisimple and $n\in\g$ is nilpotent. Let $\mathfrak{h}$ be a Cartan subalgebra of $\mathfrak{g}_s$ contained in $\ba$. The polynomials $f_1\big\vert_{\mathfrak{h}},\ldots,f_b\big\vert_{\mathfrak{h}}$ then generate $\mathbb{C}[\mathfrak{h}]^{W_s}$, where $W$ is the Weyl group of $(\mathfrak{g},\mathfrak{h})$ and $W_s$ is the $W$-stabilizer of $s$.  
\end{lem}

\begin{proof}
An application of \cite[Corollary 3.1.43]{Chriss} gives $f_i(x+\lambda a)=f_i(x+\lambda s)$ for all $i\in\{1,\ldots,r\}$, $x\in\mathfrak{h}$, and $\lambda\in\mathbb{C}$. It follows that $f_{ij}^a(x)=f_{ij}^s(x)$ for all $i\in\{1,\ldots,r\}$, $j\in\{0,\ldots,d_i-1\}$, and $x\in\mathfrak{h}$, where the $f_{ij}^a$ are defined in \eqref{eq:Def_fij} and the $f_{ij}^s$ are defined analogously. In other words, we have 
\[
f_{ij}^a\big\vert_{\mathfrak{h}}=f_{ij}^s\big\vert_{\mathfrak{h}},\quad i\in\{1,\ldots,r\},\text{ }j\in\{0,\ldots,d_i-1\}.
\] 
We are thus reduced to proving that the $f_{ij}^s\big\vert_{\mathfrak{h}}$ generate $\mathbb{C}[\mathfrak{h}]^{W_s}$. To this end, the first lemma appearing on page 302 of \cite{Shuvalov} implies that the $f_{ij}^s\big\vert_{\mathfrak{g}_s}$ generate $\mathbb{C}[\mathfrak{g}_s]^{G_s}$. Note also that $f_{ij}^s\big\vert_{\mathfrak{h}}$ is the image of $f_{ij}^s\big\vert_{\mathfrak{g}_s}$ under Chevalley's restriction isomorphism $\mathbb{C}[\mathfrak{g}_s]^{G_s}\xrightarrow{\cong}\mathbb{C}[\mathfrak{h}]^{W_s}$. It follows that the $f_{ij}^s\big\vert_{\mathfrak{h}}$ must generate $\mathbb{C}[\mathfrak{h}]^{W_s}$, completing the proof.
\end{proof}

We may now establish the following qualitative features of $F_a(\ba)$.

\begin{thm}\label{Theorem: F(ba)}
Suppose that $a\in\mathfrak{g}_{\emph{reg}}$ has a Jordan decomposition of $a=s+n$, where $s\in\g$ is semisimple and $n\in\g$ is nilpotent. Let $\pi_r:\mathbb{C}^b\rightarrow\mathbb{C}^r$ denote projection onto the first $r$ components, and let $\h$ be a Cartan subalgebra of $\g_s$ satisfying $\h\subseteq\ba$. Consider the Weyl group $W$ of $(\g,\h)$, and let $W_s\subseteq W$ be the $W$-stabilizer of $s$.
\begin{enumerate}
\item[(i)] We have $F_a(\ba)=F_a(\h)$.
\item[(ii)] The image $F_a(\ba)$ is an $r$-dimensional closed subvariety of $\mathbb{C}^b$ satisfying $\pi_r(F_a(\ba)) = \C^r$. If $a$ is nilpotent, then $F_a(\ba) = \C^r\times \{0\}\subseteq\mathbb{C}^b$.
\item[(iii)] Consider the restricted map $F_a\big\vert_{\h}:\h\rightarrow F_a(\ba)$ obtained by virtue of (i). The associated comorphism $(F_a\big\vert_{\mathfrak{h}})^*:\mathbb{C}[F_a(\ba)]\rightarrow\mathbb{C}[\mathfrak{h}]$ is injective with image $\mathbb{C}[\mathfrak{h}]^{W_s}$.
\item[(iv)] The restriction $\pi_r\big\vert_{F_a(\ba)}:F_a(\ba)\to \C^r$ is a finite morphism of degree $|W/W_s|$. 
\end{enumerate}
\end{thm}
\begin{proof} 
We begin by verifying (i). Suppose that $x\in \mathfrak{b}_{a}$ and write $x = x_{\mathfrak{h}} + x_{\ua}$ with $x_{\mathfrak{h}}\in\h$ and $x_{\ua}\in\ua:=[\ba,\ba]$. If $f\in\mathbb{C}[\mathfrak g]^G$ is an invariant polynomial, then \cite[Corollary 3.1.43]{Chriss} allows us to write
\begin{equation}\label{Equation:baCartan}
f(x+\lambda a ) =  f(x_{\mathfrak{h}}+\lambda a+x_{\ua}) = f(x_{\mathfrak{h}} + \lambda a)
\end{equation}
for all $\lambda\in\mathbb{C}$. It follows that $F_a(\ba)=F_a(\h)$.

We now verify (ii). Consider the restriction $F_a\big\vert_{\mathfrak{h}}:\mathfrak{h}\rightarrow\mathbb{C}^b$, as well as the induced map of coordinate rings $(F_a\big\vert_{\mathfrak{h}})^*:\mathbb{C}[x_1,\ldots,x_b]\rightarrow\mathbb{C}[\mathfrak{h}]$. We then have $(F_a\big\vert_{\mathfrak{h}})^*(x_i)=f_i\big\vert_{\mathfrak{h}}$ for all $i=1,\ldots,b$. The polynomials $f_1\big\vert_{\mathfrak{h}},\ldots,f_r\big\vert_{\mathfrak{h}}$ generate the subalgebra $\mathbb{C}[\mathfrak{h}]^W$, so that the image of $(F_a\big\vert_{\mathfrak{h}})^*$ must contain $\mathbb{C}[\mathfrak{h}]^W$. Since the Chevalley--Shephard--Todd theorem shows $\mathbb{C}[\mathfrak{h}]$ to be a free module of finite rank $\vert W\vert$ over $\mathbb{C}[\mathfrak{h}]^W$, it follows that $\mathbb{C}[\mathfrak{h}]$ is finitely generated over $\mathbb{C}[x_1,\ldots,x_b]$. This amounts to $F_a\big\vert_{\mathfrak{h}}$ being a finite morphism of affine varieties. Noting that finite morphisms are closed, we see that $F_a(\mathfrak{h})=F_a(\ba)$ is a closed subset of $\mathbb{C}^b$. 

The equality
$F_a(\ba) = F_a(\mathfrak h)$
implies $\dim(F_a(\ba))\leq \dim \mathfrak{h} = r$. 
On the other hand, note that $\pi_r\circ F_a: \mathfrak g\to\C^r$ is the map $F$ from \eqref{eq:AdQuot}. The restriction $F\big\vert_{\h}:\h\rightarrow\mathbb{C}^r$ is known to be surjective (see the proof of \cite[Proposition 10]{KostantLie}), so that we must have 
\[
\pi_r(F_a(\ba)) = \C^r.
\]
This implies that $\dim(F_a(\ba))\geq r$, and we conclude that $\dim(F_a(\ba)) = r$.

To prove the second claim in (ii), suppose that $a$ is nilpotent. The equation \eqref{Equation: Nilpotent equation} then implies that $$F_a\big\vert_{\h}=(F\big\vert_{\h},0,\ldots,0):\g\rightarrow\mathbb{C}^b.$$ It follows that $F_a(\ba)=F_a(\h)=F(\h)\times\{0\}=\mathbb{C}^r\times\{0\}\subseteq\mathbb{C}^b$, where the final instance of equality comes from $F\big\vert_{\h}:\h\rightarrow\mathbb{C}^r$ being surjective.

To prove (iii), we consider the restricted map $F_a\big\vert_{\h}:\h\rightarrow F_a(\ba)$. Since this map is surjective (by (i)), it induces an injection $(F_a\big\vert_{\mathfrak{h}})^*:\mathbb{C}[F_a(\ba)]\rightarrow\mathbb{C}[\h]$ of coordinate rings. At the same time, note that (ii) shows the inclusion $j:F_a(\ba)\hookrightarrow\mathbb{C}^b$ to be a closed immersion. It follows that $j^*:\mathbb{C}[x_1,\ldots,x_b]\rightarrow\mathbb{C}[F_a(\ba)]$ is surjective, so that $\mathbb{C}[F_a(\ba)]$ is generated by $j^*(x_1),\ldots,j^*(x_b)$. On the other hand, it is straightforward to verify that
\[
(F_a\big\vert_{\mathfrak{h}})^*(j^*(x_i))=f_i\big\vert_{\mathfrak{h}}
\] 
for all $i=1,\ldots,b$. We conclude that the image of $(F_a\big\vert_{\mathfrak{h}})^*$ is generated by $f_1\big\vert_{\mathfrak{h}},\ldots,f_b\big\vert_{\mathfrak{h}}$, which by Lemma \ref{Lemma: Generate} means that the image is exactly $\mathbb{C}[\mathfrak{h}]^{W_s}$.

Let us now prove (iv). Note that Chevalley's restriction theorem forces $f_1\big\vert_{\mathfrak{h}},\ldots,f_r\big\vert_{\mathfrak{h}}$ to be algebraically independent generators of $\mathbb{C}[\mathfrak{h}]^W$, so that there is a unique $\mathbb{C}$-algebra isomorphism $\mathbb{C}[x_1,\ldots,x_r]\rightarrow\mathbb{C}[\mathfrak{h}]^W$ sending $x_i$ to $f_i\big\vert_{\mathfrak{h}}$, $i=1,\ldots,r$. This isomorphism fits into the commutative diagram 
\begin{equation*}
\begin{xy}
\xymatrixcolsep{5pc}\xymatrix{
\C[x_1,\dots,x_r]   \ar[r]^{\cong} \ar[d]_{\big(\pi_r\big\vert_{F_a(\ba)}\big)^*} & \ar[d] \C[\mathfrak h]^W    \\
     \C[F_a(\ba)]  \ar[r]^{\big(F_a\big\vert_{\mathfrak h}\big)^*}_{\cong} & \C[\mathfrak h]^{W_s},   }
\end{xy}
\end{equation*}
where the rightmost vertical map is the obvious inclusion, $\big(F_a\big\vert_{\mathfrak{h}}\big)^*$ is the map from (iii), and $\big(\pi_r\big\vert_{F_a(\ba)}\big)^*$ is the map of coordinate rings corresponding to $\pi_r\big\vert_{F_a(\ba)}:F_a(\ba)\rightarrow\mathbb{C}^r$. Now consider the corresponding commutative diagram 
\begin{equation*}
\begin{xy}
\xymatrixcolsep{5pc}\xymatrix{
\mathfrak{h}/W_s  \ar[r]^{\cong} \ar[d] & \ar[d]^{\pi_r\big\vert_{F_a(\ba)}} F_a(\ba)    \\
\mathfrak{h}/W  \ar[r]_{\cong} & \mathbb{C}^r}
\end{xy}
\end{equation*}
in the category of affine varieties, noting that the map $\mathfrak{h}/W_s\rightarrow\mathfrak{h}/W$ is a finite morphism of degree $\vert W/W_s\vert$. Since the horizontal arrows in this second diagram are isomorphisms, it follows that $\pi_r\big\vert_{F_a(\ba)}$ is a finite morphism of degree $\vert W/W_s\vert$.  
\end{proof}

\section{Examples}\label{Sec:Examples} 
In the interest of concreteness, we illustrate some of our results in the cases $\g = \sln_2(\C)$ and $\g = \sln_3(\C)$. 

\subsection{The case $\g = \sln_2(\C)$}\label{subsection:sln2}
Consider $\mathfrak g =\sln_2(\C)$ and let $\h\subseteq\mathfrak{sl}_2(\mathbb{C})$ be the standard Cartan subalgebra of diagonal matrices. We also have the positive Borel subalgebra $\bb_{+}\subseteq\mathfrak{sl}_2(\mathbb{C})$ of upper-triangular matrices, as well as the negative Borel subalgebra $\bb_{-}\subseteq\mathfrak{sl}_2(\mathbb{C})$ of lower-triangular matrices. 

Observe that $a\in\mathfrak{sl}_2(\mathbb{C})$ is regular if and only if $a\neq 0$. At the same time, recall that the relevant properties of $F_a$ only depend on the conjugacy class of $a$ (see Remark \ref{Remark: Essential properties}). We will therefore assume that $a$ is one of
\[
 s= \begin{pmatrix}
 a_1 & 0 \\ 0 & -a_1
 \end{pmatrix} (a_1\neq 0) \quad \text{and}\quad
 n= \begin{pmatrix}
 0 & 1 \\ 0 & 0
 \end{pmatrix}.
\]  
Note that $\Sing^a = \C a$ in each case. 

The Killing form determines a quadratic form on $\mathfrak{sl}_2(\mathbb{C})$, and this freely generates the algebra of invariant polynomials on $\sln_2(\C)$. Note that the aforementioned quadratic form takes $x\in \sln_2(\C)$ to a constant multiple of $\mathrm{tr}(x^2)$. A straightforward calculation then justifies our taking $F_a:\sln_2(\C)\to \C^2$ to be
\[
F_a(x) = (\frac{1}{2}\mathrm{tr}(x^2), \mathrm{tr}(ax)).
\]
If we now write 
\begin{equation}\label{Equation: Matrix equation}
x= \left(\begin{array}{cc} x_1 & x_2 \\ x_3 & -x_1\end{array}\right)\in\mathfrak{sl}_2(\mathbb{C}),
\end{equation} then we obtain 
\begin{eqnarray*}
F_s(x_1,x_2,x_3) &=& (x_1^2 + x_2x_3, 2a_1x_1),\\
F_n(x_1,x_2,x_3) &=& (x_1^2 + x_2x_3, x_3).
\end{eqnarray*}
It follows that
\begin{eqnarray*}
F_s^{-1}(0) &=& \{x\in\sln_2(\C)\colon  x_1^2 + x_2x_3 =0 = 2a_1x_1\} = \{x\in\sln_2(\C)\colon x_1 =0, x_2x_3=0\} = \uu_+\cup \uu_-,\\
F_n^{-1}(0) &=& \{x\in\sln_2(\C)\colon x_1^2 + x_2x_3 = 0 = x_3\} = \uu_+,
\end{eqnarray*}
where $\uu_{+}$ and $\uu_{-}$ are the nilpotent radicals of $\bb_{+}$ and $\bb_{-}$, respectively. Note that $F_a^{-1}(0)$ has no exotic irreducible components in each case. Using the notation of Section \ref{Subsection: A recursive formula}, these last two sentences imply that $|^{\sln_2(\C)}\mathcal I_s| = 2$, $|^{\sln_2(\C)}\mathcal I_n| = 1$, and $|^{\sln_2(\C)}\mathcal I_s'| = 0 = |^{\sln_2(\C)}\mathcal I_n'|$. One can also establish that $\mathcal B_s = \{\bb_{+},\bb_{-}\}$, $\mathcal B_n = \{\bb_{+}\}$, and $\mathcal P_s = \mathcal P_n  =\emptyset$. The reader will note that the previous two sentences are consistent with our recursive formula \eqref{Equation:RecursiveFormula}.

We now describe the images $F_a(\Sing^a)$ and $F(\bb^a)$, treating the cases $a=s$ and $a=n$ separately. To this end, suppose that $a=s$ and $z=(z_1,z_2)\in\mathbb{C}^2$. Let us write $T$ for the maximal torus of diagonal matrices in $\SLn_2(\C)$, in which case we have following:
\[
F_s^{-1}(z) = \left\{\begin{pmatrix}
\tfrac{z_2}{2a_1} & x_2 \\\left(z_1-\tfrac{z_2^2}{4a_1^2}\right)\tfrac{1}{x_2} & -\tfrac{z_2}{2a_1}
\end{pmatrix} \colon x_2\in \C\setminus\{0\} \right\}  = T\begin{pmatrix}
\tfrac{z_2}{2a_1} & 1 \\z_1-\tfrac{z_2^2}{4a_1^2} & -\tfrac{z_2}{2a_1}
\end{pmatrix} \quad \text{if} \quad z_1-\tfrac{z_2^2}{4a_1^2}\neq 0 
\]
and
\[
F_s^{-1}(z) = \left\{\begin{pmatrix} \tfrac{z_2}{2a_1} & x_2 \\ 0 & -\tfrac{z_2}{2a_1}\end{pmatrix}\colon x_2\in\C \right\} \cup \left\{\begin{pmatrix} \tfrac{z_2}{2a_1} & 0 \\ x_3 & -\tfrac{z_2}{2a_1}\end{pmatrix} \colon x_3\in\C\right\} \quad \text{if}\quad z_1-\tfrac{z_2^2}{4a_1^2} = 0.
\]
If $z_1-\tfrac{z_2^2}{4a_1^2}\neq 0$, then $F_s^{-1}(z)$ does not meet $\Sing^s = \h$ and is irreducible. If $z_1-\tfrac{z_2^2}{4a_1^2}=0$, then the irreducible components of $F_s^{-1}(z)$ are $x +\uu_{+}$ and $x+\uu_{-}$ with $x = \tfrac{z_2}{2a_1^2}s\in\Sing^s$. This shows that $F_s^{-1}(z)$ meets $\Sing^s$ if and only if $z_1-\tfrac{z_2^2}{4a_1^2}=0$. Noting that $\Sing^s  = \bb^s = \h$, we have 
\begin{equation}\label{Equation: First example} F_s(\Sing^s)= F_s(\bb^s) = \left\{(z_1,z_2)\in \C^2\colon z_1-\tfrac{z_2^2}{4a_1^2} = 0\right\}.\end{equation} 
Observe that projection onto the first factor $F_s(\bb_s)\rightarrow\mathbb{C}$ is two-to-one, except over the origin (cf. Theorem \ref{Theorem: F(ba)}(iv)).

In the case $a=n$, we have 
\[
F_n^{-1}(z) = \{x\in \sln_2(\C)\colon x_1^2 - x_2z_2 = z_1,\; x_3 = z_2\}
\]
for all $z=(z_1,z_2)\in\mathbb{C}^2$.
We see that $F_n^{-1}(z)$ is irreducible if $z_2\neq 0$ or $z_1=0=z_2$, and we have already noted that $F_n^{-1}(0) = \uu_{+} = \C n = \Sing^n$ in the latter case. If $z_2=0\neq z_1$, then denote by $\sqrt{z_1}$ a fixed square root of $z_1$. We then have 
$$F_n^{-1}(z) = \big(\mathrm{diag}\left(\sqrt{z_1}, -\sqrt{z_1}\right)+\uu_{+}\big) \cup \big(\mathrm{diag}(-\sqrt{z_1}, \sqrt{z_1}) + \uu_+\big).$$  This consists of two irreducible components, each contained in $\bb_{+}$ (cf. Proposition \ref{Prop: Unipotent orbit}(iii)).

We now compute $F_n(\bb^n)$ and $F_n(\Sing^n)$, noting that $\bb^n = \bb = \{x_3=0\}$ in the notation \eqref{Equation: Matrix equation}. Since $F_n(x_1,x_2,0) =  (x_1^2, 0)$, we see that
$$F_n(\bb^n) = \left\{(z_1,z_2)\in\C^2\colon z_2=0\right\} = \C\times\{0\}.$$ 
We also have $\Sing^n = \C n = \{x_1=0=x_3\} = \uu_+$, yielding 
\begin{equation}\label{Equation: Second example}
F_n(\Sing^n) = \{0\}.
\end{equation}

\begin{rem}\label{Remark: Possible codimensions}
Recall that Proposition \ref{Proposition:dimF(Sing)} reduces the possible codimensions of $\overline{F_a(\Sing^a)}$ to $1$ and $2$. Equations \eqref{Equation: First example} and \eqref{Equation: Second example} show that each of these possible codimensions is achievable.
\end{rem}

\subsection{The case $\g = \sln_3(\C)$} \label{Subsection: The case g=sln3}
Suppose that $\mathfrak g = \sln_3(\C)$ and that $a\in\mathfrak{sl}_3(\mathbb{C})_{\text{reg}}$.
The map $F_a:\mathfrak{sl}_3(\mathbb{C})\rightarrow\mathbb{C}^5$ is more complicated than its $\mathfrak{sl}_2(\mathbb{C})$ counterpart, to the point that general fibres of $F_a$ are more difficult to describe explicitly. We will nevertheless show that $F_a^{-1}(0)$ has an exotic irreducible component, and we will illustrate the assertions of Theorem \ref{Theorem: F(ba)}. 

We may take $F_a:\sln_3(\C)\to \C^5$ to be given by
\begin{equation}\label{Equation: Specific MF}
F_a(x) = (\mathrm{tr}(x^2), \mathrm{tr}(x^3), 2\mathrm{tr}(ax),3\mathrm{tr}(ax^2),6\mathrm{tr}(a^2x)).
\end{equation}
It is also straightforward to verify that
\[  
s= \begin{pmatrix} s_1 & 0 & 0 \\0 & s_2 & 0 \\ 0 & 0 & s_3\end{pmatrix} (s_i\neq s_j\text{ }\forall i\neq j),\quad r = \begin{pmatrix} \rho & 1 & 0\\0 & \rho & 0\\0 &0 & -2\rho\end{pmatrix}(\rho\neq 0),\quad\text{and}\quad n = \begin{pmatrix} 0 & 1 & 0 \\ 0 & 0 & 1\\ 0 & 0 & 0\end{pmatrix}
\]
are a complete collection of representatives for the conjugacy classes of regular elements in $\mathfrak{sl}_3(\mathbb{C})$. In light of Remark \ref{Remark: Essential properties}, we will assume that $a$ is one of $s$, $r$, and $n$.

Recalling the notation established in Section \ref{Subsection: A recursive formula}, we obtain the following data.
$$\mathcal B_s = \left \{  \scalemath{0.8}{
\begin{pmatrix}
* & * & * \\ 0 & * & 
 \\ 0 & 0 & *
\end{pmatrix}, 
\begin{pmatrix}
* & 0 & 0 \\ * & * & 0 \\ * & * & *
\end{pmatrix}, 
\begin{pmatrix}
* & * & 0 \\ 0 & * & 0 \\ * & * & *
\end{pmatrix}, 
\begin{pmatrix}
* & 0 & * \\ * & * & * \\ 0 & 0 & *
\end{pmatrix}, 
\begin{pmatrix}
* & 0 & 0 \\ * & * & * \\ * & 0 & *
\end{pmatrix}, 
 \begin{pmatrix}
* & * & * \\ 0 & * & 0 \\ 0 & * & *
\end{pmatrix} }   
\right\},$$
$$\mathcal{P}_s = \left \{ 
\p_1 = \scalemath{0.8}{\begin{pmatrix}
* & * & * \\ * & * & * \\ 0 & 0 & *
\end{pmatrix}},
\p_1^- = \scalemath{0.8}{\begin{pmatrix}
* & * & 0 \\ * & * & 0 \\ * & * & *
\end{pmatrix}}, 
\p_2 = \scalemath{0.8}{\begin{pmatrix}
* & * & * \\ 0 & * & * \\ 0 & * & *
\end{pmatrix}},  
\p_2^- = \scalemath{0.8}{\begin{pmatrix}
* & 0 & 0 \\ * & * & * \\ * & * & *
\end{pmatrix}}, 
\p_3 = \scalemath{0.8}{\begin{pmatrix}
* & * & * \\ 0 & * & 0 \\ * & * & *
\end{pmatrix}}, 
\p_3^- = \scalemath{0.8}{\begin{pmatrix}
* & 0 & * \\ * & * & * \\ * & 0 & *
\end{pmatrix} }    
\right \},$$ 

$$\mathcal B_r = \left \{ 
\scalemath{0.8}{\begin{pmatrix}
* & * & * \\ 0 & * & * \\ 0 & 0 & *
\end{pmatrix}, 
\begin{pmatrix}
* & * & 0 \\ 0 & * & 0 \\ * & * & *
\end{pmatrix},  
\begin{pmatrix}
* & * & * \\ 0 & * & 0 \\ 0 & * & *
\end{pmatrix}}     
\right\},$$

$$\mathcal{P}_r = \left \{ 
\p_1 = \scalemath{0.8}{\begin{pmatrix}
* & * & * \\ * & * & * \\ 0 & 0 & *
\end{pmatrix}},
\p_1^- = \scalemath{0.8}{\begin{pmatrix}
* & * & 0 \\ * & * & 0 \\ * & * & *
\end{pmatrix}}, 
\p_2 = \scalemath{0.8}{\begin{pmatrix}
* & * & * \\ 0 & * & * \\ 0 & * & *
\end{pmatrix}},  
\p_3 = \scalemath{0.8}{\begin{pmatrix}
* & * & * \\ 0 & * & 0 \\ * & * & *
\end{pmatrix}},      
\right \},$$ 
 
$$\mathcal B_n = \left \{ 
\scalemath{0.8}{\begin{pmatrix}
* & * & * \\ 0 & * & * \\ 0 & 0 & *
\end{pmatrix}}   
\right \},$$
$$\mathcal{P}_n = \left \{ 
\p_1 = \scalemath{0.8}{\begin{pmatrix}
* & * & * \\ * & * & * \\ 0 & 0 & *
\end{pmatrix}}, 
\p_2 = \scalemath{0.8}{\begin{pmatrix}
* & * & * \\ 0 & * & * \\ 0 & * & *
\end{pmatrix}}  
\right \}.$$  

In what follows, $\h\subseteq\mathfrak{sl}_3(\mathbb{C})$ is the usual Cartan subalgebra of diagonal matrices and $\uu_{+}\subseteq\mathfrak{sl}_3(\mathbb{C})$ and $\uu_{-}\subseteq\mathfrak{sl}_3(\mathbb{C})$ are the maximal nilpotent subalgebras of upper-triangular and lower-triangular matrices, respectively. 

\begin{prop}\label{Proposition: Calculation}
If $s\in\sln_3(\C)$ is as defined above, then there exists an element $x\in F_s^{-1}(0)\cap (\uu_-\oplus\uu_+)$ that is not contained in any $\p\in\widetilde{\mathcal{P}_s}$.
\end{prop}
\begin{proof}
Given $i\in\{1,2,3\}$, let $\epsilon_i\in\mathfrak{h}^*$ be the linear functional that picks out the diagonal entry in position $(i,i)$. We then have the standard simple roots $\alpha :=\alpha_1= \epsilon_1 - \epsilon_2$ and  $\beta :=\alpha_2= \epsilon_2-\epsilon_3$, and we put $\gamma :=\alpha_3= \alpha + \beta = \epsilon_1-\epsilon_3$. It follows that $\Delta=\{\alpha,\beta,\gamma,-\alpha,-\beta,-\gamma\}$ and $\Delta_+ = \{\alpha,\beta,\gamma\}$. 

If $i,j\in\{1,2,3\}$ are distinct, let $e_{ij}\in\mathfrak{sl}_3(\mathbb{C})$ be the matrix with $1$ in position $(i,j)$ and all remaining entries equal to $0$. Let us set
$$e_\alpha := e_{12},\quad e_\beta := e_{23},\quad e_\gamma := e_{13},\quad e_{-\alpha} := e_{21},\quad e_{-\beta} := e_{32},\quad e_{-\gamma} := e_{31},$$
reflecting the fact that $e_{12}$ lies in the $\alpha$-root space, $e_{23}$ is in the $\beta$-root space, etc. We also consider the matrices in $\h$ given by 
$$h_\alpha := [e_\alpha,e_{-\alpha}] = \mathrm{diag}(1,-1,0)\quad\text{and}\quad h_\beta := [e_\beta,e_{-\beta}] = \mathrm{diag}(0,1,-1).$$  
Now expand $s$ in the basis $\{h_\alpha,h_\beta\}$ of $\h$, i.e. $$s = s_\alpha h_\alpha + s_\beta h_\beta=\mathrm{diag}(s_\alpha,s_\beta-s_\alpha,-s_\beta)$$
for $s_{\alpha},s_{\beta}\in\mathbb{C}$. 
Note that 
\begin{equation}\label{Equation:alphasbetas}
\alpha(s) = 2s_\alpha-s_\beta,\quad\text{and}\quad\beta(s) = 2s_\beta-s_\alpha.
\end{equation} 

Now consider an element 
$$x = \sum_{\nu\in\Delta} x_\nu e_\nu \in\uu_-\oplus \uu_+,$$
 where all $x_\nu\in\C$. Let us choose $x_\alpha, x_{\beta}, x_{-\gamma}\in \C$ such that  
$$(x_{\alpha}x_{\beta}x_{-\gamma})^2 = \alpha(s)\beta(s)\gamma(s).$$ Note that $x_{\alpha}$, $x_{\beta}$, and $x_{-\gamma}$ are non-zero, as $s$ being a regular element forces each of $\alpha(s)$, $\beta(s)$, and $\gamma(s)$ to be non-zero. Now define $x_{-\alpha}, x_{-\beta}, x_{\gamma}\in\C$ by the conditions 
$$x_{\alpha}x_{-\alpha} = \alpha(s),\quad x_{\beta}x_{-\beta} = \beta(s), \quad x_{\gamma}x_{-\gamma} = -\gamma(s).$$
It follows that $x_\nu\neq 0$ for all $\nu\in\Delta$. A glance at the above-listed elements of $\widetilde{\mathcal P_s}=\mathcal{B}_s\cup\mathcal{P}_s$ then reveals that $x\notin \p$ for all $\p\in\widetilde{\mathcal P_s}$. 

It remains only to verify that $x$ solves $F_s(x) = 0$. To this end, recall the form of our Mishchenko--Fomenko map $F_s:\mathfrak{sl}_3(\mathbb{C})\rightarrow\mathbb{C}^5$. We have $\tr(s^2x) = 0 = \tr(sx)$, so that the equation $F_s(x) = 0$ is equivalent to the system of equations 
\begin{eqnarray*}
x_{\alpha}x_{-\alpha} + x_{\beta}x_{-\beta} + x_{\gamma}x_{-\gamma} &=& 0\qquad (\iff \tr(x^2) = 0)\\
s_\beta x_{\alpha}x_{-\alpha} - s_\alpha x_{\beta}x_{-\beta} - (s_\beta-s_\alpha)x_{\gamma}x_{-\gamma} &=& 0 \qquad (\iff \tr(sx^2) = 0)\\
x_{\alpha}x_{\beta}x_{-\gamma} + x_{\gamma}x_{-\beta}x_{-\alpha} &=& 0 \qquad (\iff \tr(x^3) = 0).
\end{eqnarray*}

To address the first equation, note that 
$$x_{\alpha}x_{-\alpha} + x_{\beta}x_{-\beta} + x_{\gamma}x_{-\gamma} = \alpha(s) +\beta(s) -\gamma(s) = 0.$$ The second equation is also satisfied, as \eqref{Equation:alphasbetas} gives 
\begin{eqnarray*}
s_\beta x_{\alpha}x_{-\alpha} - s_\alpha x_{\beta}x_{-\beta} - (s_\beta-s_\alpha)x_{\gamma}x_{-\gamma} &=& s_\beta\alpha(s) - s_\alpha\beta(s)+(s_\beta-s_\alpha)\gamma(s)\\
&=& s_\beta\alpha(s) - s_\alpha\beta(s) + (s_\beta-s_\alpha)(\alpha(s)+\beta(s))\\
&=& 2s_\alpha s_\beta - s_\beta^2 - 2s_\alpha s_\beta + s_\alpha^2 + (s_\beta -s_\alpha)(s_\alpha+s_\beta)\\
&=& s_\alpha^2-s_\beta^2 + s_\beta^2 - s_\alpha^2\\
&=& 0.
\end{eqnarray*}
To deal with the third equation, observe that 
$$(x_\alpha x_\beta x_{-\gamma}) \left(x_{\alpha}x_{\beta}x_{-\gamma} + x_{\gamma}x_{-\beta}x_{-\alpha}\right) = (x_\alpha x_\beta x_{-\gamma})^2 + x_\alpha x_{-\alpha} x_\beta x_{-\beta} x_\gamma x_{-\gamma} = \alpha(s)\beta(s)\gamma(s) - \alpha(s)\beta(s)\gamma(s) =0$$
Since $x_\alpha x_\beta x_{-\gamma} \neq 0$, this implies 
$$x_{\alpha}x_{\beta}x_{-\gamma} + x_{\gamma}x_{-\beta}x_{-\alpha} = 0.$$
\end{proof}

\begin{rem}
Recall the notation used in the proof of Proposition \ref{Proposition: Calculation}. One can verify that the element $x=\sum_{\nu} x_\nu e_\nu$ is regular nilpotent, i.e. $x^2 \neq 0 = x^3$. We also have $$\tr(sx) = \tr(s^2x) = \tr(sx^2)=\tr(x^2) = \tr(x^3) = 0,$$ so that $$\tr((x+\lambda s)^2) = \tr((\lambda s)^2)\quad\text{and}\quad \tr((x+\lambda s)^3) = \tr((\lambda s)^3).$$ This is easily seen to imply that $x+\lambda s$ and $\lambda s$ belong to the same adjoint orbit for all $\lambda\neq 0$. Since $x$ and $s$ are both regular, we conclude that $x\in\mathfrak{sl}_3(\mathbb{C})_{\text{sreg}}^s$. 
\end{rem}

\begin{prop}\label{Proposition: Case r}
If $r\in\sln_3(\C)$ is as defined above, then there exists an element $x\in F_r^{-1}(0)\cap \mathrm{image}(\ad_r)$ that is not contained in any $\p\in\widetilde{\mathcal{P}_r}$.
\end{prop}

\begin{proof}
Let us write
$$x=\begin{pmatrix}
x_{11} & x_{12} & x_{13}\\ x_{21} & x_{22} & x_{23} \\ x_{31} & x_{32} & x_{33}
\end{pmatrix}\in\mathfrak{sl}_3(\mathbb{C}).$$
Now recall the definition of $r$ and observe that $x\in\mathrm{image}(\ad_r)$ if and only if $x_{11} + x_{22} = 0 = x_{33} = x_{21}$. In this case, the equation $F_r(x) = 0$ reduces to the system
\begin{eqnarray*}
 x_{11}^2 + x_{13}x_{31} + x_{23}x_{32} &=& 0\\
x_{11}(x_{13}x_{31}-x_{23}x_{32}) + x_{12}x_{23}x_{31} &=& 0\\ 
2\rho x_{11}^2 - \rho x_{13}x_{31} - \rho x_{23}x_{32}  + x_{23}x_{31} &=& 0.
\end{eqnarray*}
One can use a direct calculation to check that 
$$x= \begin{pmatrix}
-3\rho & 1 & 3\rho i\\ 0 & 3\rho & 9\rho^2 i \\ 3\rho i & 0 & 0
\end{pmatrix}
$$ is a solution, where $i=\sqrt{-1}$.
An inspection of the list $\widetilde{\mathcal P_r}=\mathcal{B}_r\cup\mathcal{P}_r$ shows that $x$ is not contained in any $\p\in\widetilde{\mathcal P_r}$, since $x_{23},x_{31} \neq 0$.
\end{proof}

\begin{rem}
A direct calculation establishes that the above-constructed matrix 
$$x= \begin{pmatrix}
-3\rho & 1 & 3\rho i\\ 0 & 3\rho & 9\rho^2 i \\ 3\rho i & 0 & 0
\end{pmatrix}
$$
is regular nilpotent, i.e. $x^2 \neq 0=x^3$. One can also check that $x+\lambda r$ and $\lambda r$ have the same Jordan canonical form for all $\lambda\neq 0$. It follows that $x+\lambda r$ and $\lambda r$ are $\operatorname{SL}_3(\mathbb{C})$-conjugate for all $\lambda\neq 0$, and we conclude that $x\not\in\mathrm{Sing}^r$.
\end{rem}

It remains to consider the case $a=n$. To this end, Charbonnel and Moreau establish that 
\begin{equation}\label{Equation:CM}
x= \begin{pmatrix}
0 & 0 & 0\\ 1 & 0 & 0 \\0 & -1 & 0
\end{pmatrix}
\end{equation}
is contained in $F_n^{-1}(0)$. It is also clear that $x$ does not lie in any $\p\in\widetilde{\mathcal{P}_n}$. This combines with Propositions \ref{Proposition: Calculation} and \ref{Proposition: Case r} to yield the following result.

\begin{thm}\label{Theorem: Exotic component}
If $a\in\sln_3(\C)_{\emph{reg}}$, then $F_a^{-1}(0)$ has a point that does not lie in any $\p\in\widetilde{\mathcal{P}_a}$. Equivalently, $F_a^{-1}(0)$ has an exotic irreducible component.
\end{thm}

\begin{rem}
Note that $F_a$ is invariant under the action of the centralizer $\SLn_3(\C)_a\subseteq \SLn_3(\C)$. Note also that $F_a^{-1}(0)$ is invariant under the dilation action of $\mathbb{C}^{\times}$ on $\mathfrak{sl}_3(\mathbb{C})$, and that this action commutes with the adjoint action of $\SLn_3(\C)_a$ on $\mathfrak{sl}_3(\mathbb{C})$. Now let $x$ be the element constructed in the proof of Proposition \ref{Proposition: Calculation} if $a=s$, Proposition \ref{Proposition: Case r} if $a=r$, and Equation \eqref{Equation:CM} if $a=n$. It is then straightforward to deduce that the exotic irreducible component referenced in Theorem \ref{Theorem: Exotic component} is the orbit closure $\overline{(\C^{\times}\times\SLn_3(\C)_a)x}\subseteq F_a^{-1}(0)$. 
\end{rem}

We now illustrate Theorem \ref{Theorem: F(ba)}, and this involves considering the subalgebra $\bb^a$ in each of the cases $a=s$, $a=r$, and $a=n$. A first observation is that 
\[
\bb^s = \begin{pmatrix} * & 0 & 0 \\ 0 & * & 0\\ 0 & 0 & *\end{pmatrix} = \h, \quad 
\bb^r =  \begin{pmatrix} * & * & 0 \\ 0 & * & 0\\ 0 & 0 & *\end{pmatrix}, \quad\text{and}\quad \bb^n =  \begin{pmatrix} * & * & * \\ 0 & * & *\\ 0 & 0 & *\end{pmatrix}.
\]
Recall also that the Weyl group $W$ is generated by the simple reflections $\sigma_1: \mathrm{diag}(x_{11},x_{22},-x_{11}-x_{22})\mapsto \mathrm{diag}(x_{22},x_{11},-x_{11}-x_{22})$ and $\sigma_2:\mathrm{diag} (x_{11},x_{22},-x_{11}-x_{22})\mapsto \mathrm{diag}(-x_{22}-x_{11},x_{22},x_{11})$.

A straightforward computation reveals that the restriction of $F_s$ to $\bb^s = \h$ is
\[
F_s(x) =\left(\begin{array}{c}2(x_{11}^2 + x_{22}^2 + x_{11}x_{22})\\
 -3(x_{11}^2x_{22} + x_{11}x_{22}^2)\\ 
2(2s_1+ s_2)x_{11} + 2(s_1 + 2s_2)x_{22}\\ 
-3(s_2x_{11}^2 + 2(s_1 + s_2)x_{11}x_{22} + s_1x_{22}^2)\\
-6(2s_1s_2 + s_2^2)x_{11} - 6(s_1^2 + 2s_1s_2)x_{22}
\end{array}\right),
\]
where $x_{ij}$ denotes the entry of $x\in\bb^s$ in position $(i,j)$.
One can also show that the restriction of $F_r$ to $\bb^r$ is given by
\[
F_r(x) =\left(\begin{array}{c}2(x_{11}^2 + x_{22}^2 + x_{11}x_{22})\\
 -3(x_{11}^2x_{22} + x_{11}x_{22}^2)\\ 
6\rho(x_{11} + x_{22})\\
-3\rho (x_{11}^2  +4x_{11}x_{22} +x_{22}^2)
\\
-18\rho^2(x_{11} +x_{22})
\end{array}\right),\quad x\in\bb^r.
\]
The restriction of $F_n$ to $\bb^n = \bb$ is given by
\[
F_n(x) =\left(\begin{array}{c}2(x_{11}^2 + x_{22}^2 + x_{11}x_{22})\\
 3( - x_{11}^2x_{22}  - x_{11}x_{22}^2)\\ 
0\\
0\\
0
\end{array}\right),\quad x\in\bb^n.
\]
Note that each of these restrictions only depends on the diagonal part $x_{\h}$ of $x$ (cf. Theorem \ref{Theorem: F(ba)}(i)). 
We also see that $F_n(\bb^n) = \C^2\times\{0\}$, which is consistent with Theorem \ref{Theorem: F(ba)}(ii).

We now illustrate Theorem \ref{Theorem: F(ba)}(iv) by computing the degree of $\pi_2:F_r(\ba)\rightarrow\mathbb{C}^2$. Consider the semisimple part $r_{\h}:= \mathrm{diag}(\rho,\rho,-2\rho)$ of $r$, observing that $W_{r_{\h}}=\{\mathrm{id},\sigma_1\}$. The restriction $F_r\big\vert_{\h}:\h\rightarrow\mathbb{C}^5$ is easily seen to be invariant under $W_{r_{\h}}$. At the same time, one can verify the following fact: if $x\in\h\cap\greg$, then
\[
F_r(\sigma(x)) \neq F_r(x)
\]
unless $\sigma\in W_{r_{\h}}$. Using these last three sentences and recalling that $F_r(\ba)=F_r(\h)$ (see Theorem \ref{Theorem: F(ba)}(i)), it is straightforward to deduce that $\pi_2:  F_r(\bb^r)\to \C^2$ has degree $|W/W_{r_{\h}}|=3$.

\section*{Notation}
\begin{itemize}
\item $|M|$ cardinality of a set $M$
\item $f\big\vert_Z$ restriction of a map $f$ to a subset $Z$
\item $df(x)$ differential of a map $f$ at a point $x$
\item $Hx$ orbit of $x$ under a group $H$
\item $H_x\subseteq H$ the $H$-stabilizer of $x$
\item $\g$ semisimple Lie algebra
\item $r = \mathrm{rk}(\g)$ rank of $\g$
\item $b = b(\g)$ dimension of a Borel subalgebra of $\g$
\item $\langle\cdot,\cdot\rangle$ Killing form on $\g$
\item $G$ adjoint group of $\g$
\item $\exp:\g\rightarrow G$ exponential map
\item $\Ad$, $\ad$ adjoint representations of $G$, $\g$, respectively
\item $\g_x\subseteq\g$ the $\g$-centralizer of $x\in\g$
\item $\h\subseteq\g$ Cartan subalgebra; $T = \exp(\h)$
\item $\Delta,\Delta_+,\Pi$ sets of roots, positive roots, simple roots, respectively 
\item $W = W(\g,\h) = N_G(T)/T$ Weyl group
\item $B\subseteq G$ Borel subgroup; $\bb=\mathrm{Lie}(B)$
\item $\ba$ intersection of all Borel subalgebras of $\g$ that contain $a\in\g_{\text{reg}}$
\item $\ua$ nilpotent radical of $\ba$
\item $P\subseteq G$ parabolic subgroup; $\p=\mathrm{Lie}(P)$
\item $\lf$ Levi factor in a parabolic subalgebra
\item $\greg$ regular elements in $\g$
\item $\gsing=\g\setminus\greg$ singular elements in $\g$
\item $\Sing^a = \gsing +\C a$
\item $\gsreg = \g\setminus \Sing^a$
\item $\mathbb{C}[X]$ coordinate algebra of a complex affine variety $X$
\item $\mathbb{C}[\g]^G\subseteq\mathbb{C}[\g]$ algebra of $G$-invariant polynomials on $\g$
\item $f_1,\ldots,f_r$ homogeneous, algebraically independent generators of $\mathbb{C}[\g]^G$
\item $F$ the adjoint quotient map $(f_1,\ldots,f_r):\g\to\C^r$
\item $F_a:\g\to\C^b$ Mishchenko--Fomenko map associated with $a\in\greg$ 
\item $\mathcal F_a\subseteq \C[\g]$ Mishchenko--Fomenko subalgebra associated with $a\in\mathfrak{g}_{\text{reg}}$ 
\item $\mathcal B_a$ set of all Borel subalgebras of $\g$ that contain $a\in\greg$
\item $\widetilde{\mathcal{P}_a}$ set of all proper parabolic subalgebras of $\g$ that contain $a\in\greg$
\item $\mathcal{P}_a=\widetilde{\mathcal{P}_a}\setminus\mathcal{B}_a$
\item $^{\g}\mathcal I_a$ set of all irreducible components in $F_a^{-1}(0)$
\item $^{\g}\mathcal I_a'$ set of all $Z\in {^{\g}\mathcal I_a}$ satisfying $Z\nsubseteq\p$ for all $\p\in\mathcal{P}_a$
\end{itemize}
\bibliographystyle{acm} 
\bibliography{NewMF}
\end{document}